\documentclass{amsart}%
\usepackage{amsmath}
\usepackage{amssymb}
\usepackage{amscd}
\usepackage{amsfonts}
\usepackage{enumerate}
\usepackage{mathrsfs}
\usepackage{xy}
\usepackage{stmaryrd}
\usepackage{graphicx}
\usepackage{fancybox}
\usepackage{color}
\usepackage{multirow}
\usepackage{exscale}
\usepackage{enumitem,array}%
\usepackage{subfigure}
\usepackage{extarrows}
\usepackage{amsfonts}

\newtheorem{mytheo}{Theorem}[section]
\newtheorem{myexp}{Problem}[section]

\newtheorem{lem}[mytheo]{Lemma}
\newtheorem{preli}[mytheo]{Preliminaries}

\newtheorem{rem}[mytheo]{Remark}

\newcounter{remark}

\newcounter{problem}

\makeatletter
\def\@upcite#1#2{\textsuperscript{[{#1\if@tempswa , #2\fi}]}}
\makeatother

\title{Exponential integrators preserving first integrals or Lyapunov functions for conservative or dissipative systems }

\def\shortTitle{ENERGY PRESERVING EXPONENTIAL INTEGRATORS}

\def\myAMS{65L04, 65L05, 65M20, 65P10, 65Z05}

\def\myAbstract{
In this paper, combining the ideas of exponential integrators and discrete gradients,
we propose and analyze a new structure-preserving exponential scheme for the conservative or dissipative system $\dot{y} = Q(M y + \nabla U (y))$, where $Q$ is a $d\times d$ skew-symmetric or negative semidefinite real
matrix, $M$ is a $d\times d$ symmetric real matrix, and $U : \mathbb{R}^d\rightarrow\mathbb{R}$ is a differentiable function. We present
two properties of the new scheme. The paper is accompanied by numerical results that demonstrate
the remarkable superiority of our new scheme in comparison with other structure-preserving schemes
in the scientific literature.}

\begin{document}

%\allowdisplaybreaks[4]
%\renewcommand{\baselinestretch}{1.1}
\bibliographystyle{amsplain}
\author[]{Yu-Wen Li}
\address{Department of Mathematics, Nanjing University, Nanjing 210093,
P.R.China}
\email{farseer1118@sina.cn}
    %\thanks{\yulShortAuthor : \yulThanks}
    
\author[]{Xinyuan Wu}
\address{Department of Mathematics, Nanjing University; State Key Laboratory
for Novel Software Technology at Nanjing University, Nanjing 210093,
P.R.China}
\email{xywu@nju.edu.cn}
\subjclass[2010]{Primary \myAMS}
\date{June 29, 2016}
    %\dedicatory{}
\begin{abstract}\myAbstract\end{abstract}
    %\short{TBD}
\maketitle
\markboth{YU-WEN LI AND XINYUAN WU}{\shortTitle}

%=============================================================================================
\section{Introduction}
{The IVP}

\begin{equation}\label{ODE}\begin{aligned}
\dot{y(t)}&=Ay(t)+f(y(t)),\quad {y(t_{0})=y^{0},}
\end{aligned}
\end{equation}
arises most frequently in a variety of applications such as
mechanics, molecular dynamics, quantum physics, circuit simulations
and engineering, where $f: \mathbb{R}^{d}\rightarrow \mathbb{R}^{d}$
and $\cdot$ denotes the derivative operator $\frac{d}{dt}$. An
algorithm for \eqref{ODE} is an exponential integrator if it
{involves the computation of matrix exponential (or related matrix
functions) } and integrates the linear system
\begin{equation*}
\dot{y(t)}-Ay(t)=0
\end{equation*}
exactly. In general, exponential integrators permit larger stepsizes
and achieve higher accuracy than non-exponential ones when
\eqref{ODE} is a very stiff differential equation such as highly
oscillatory ODEs and semi-discrete time-dependent PDEs. Therefore,
numerous exponential algorithms have been proposed {for first-order
{(see, e.g.
{\cite{Berland2006,Cox2002,Hersch1958,Hochbruck1998,Hochbruck2010,Kassam2005,Klein2008,Lawson1967,Pavlov}})}
and second-order (see e.g.
\cite{Deuflhard1979,Franco2002,Gautschi1961,Hairer,Yang2009})} {
ODEs.} {On the other hand, \eqref{ODE} might inherit many important
geometrical/physical structures. For example, the canonical
Hamiltonian system
\begin{equation}\label{ODE2}\begin{aligned}
\dot{y(t)}&=J^{-1}\nabla H(y(t)),\quad {y(t_{0})=y^{0},}
\end{aligned}
\end{equation}
is a special case of \eqref{ODE}, with
\begin{equation*}
J=\left(\begin{array}{cc}O_{d_{1}\times d_{1}}&I_{d_{1}\times d_{1}}\\-I_{d_{1}\times d_{1}}&O_{d_{1}\times d_{1}}\end{array}\right).
\end{equation*}
And the flow of \eqref{ODE2} preserves the symplectic 2-form
$dy\wedge Jdy$ and the function $H(y)$. In the sense of geometric
integration, it is a natural idea to design numerical schemes that
preserve the two structures.} As far as we know, most of research
papers dealing with exponential integrators up to now focus on the
constructions of high-order explicit schemes and fail to {be
structure-preserving} except for
symmetric/symplectic/energy-preserving methods for first-order ODEs
in \cite{Celledoni2008,Cieslinski} and oscillatory second-order ODEs
(see, e.g. \cite{Hairer,Wang2012,Wu2012}). To combine ideas of
exponential integrators and energy-preserving methods, we address
ourselves to the system:
\begin{equation}\label{LNODE}
\begin{aligned}
\dot{y}=Q(My+\nabla U(y)),\quad {y(t_{0})=y^{0},}
\end{aligned}
\end{equation}
{where $Q$ is a $d\times d$ real matrix, $M$ is a $d\times d$
symmetric real matrix and $U : \mathbb{R}^{d}\rightarrow\mathbb{R}$
is a differentiable function. Clearly, \eqref{LNODE} could be
considered as a special class of \eqref{ODE} or the generalization
of \eqref{ODE2}. {However,  \eqref{LNODE} concentratively exhibits
some important structures which should be respected by a
structure-preserving scheme.} Since $M$ is symmetric, $My+\nabla
U(y)$ is the gradient of the function
$H(y)=\frac{1}{2}y^{\intercal}My+U(y)$.} If $Q$ is skew-symmetric,
then \eqref{LNODE} is a conservative system {with the first integral
$H$,} i.e. $H(y(t))$ is constant; If $Q$ is negative semi-definite
(denoted by $Q\leq0$), then \eqref{LNODE} is a dissipative system
{with the Lyapunov function $H$,} i.e. $H(y(t))$ is monotonically
decreasing. In these two cases, $H$ is also called `energy'. It
should be noted that the choice for $A$ in \eqref{ODE} or $M$ in
\eqref{LNODE} is not unique. General speaking, exponential
integrators deal with systems having a major linear term and a
comparably small nonlinear term, i.e. $||A||\gg||\frac{\partial
f}{\partial y}||$. Thus, in order to take advantage of exponential
integrators, the matrix $M$ in \eqref{LNODE} should be chosen such
that $||QM||\gg ||Q Hess(U)||$, where $Hess(U)$ is the Hessian
matrix of $U$. For example, highly oscillatory Hamiltonian systems
can be characterized by \eqref{LNODE} with a dominant linear part,
where $M$ implicitly contain the large frequency component. Up to
now, many energy-preserving or energy-decaying methods have been
proposed in the case of $M=0$ (see, e.g.
\cite{Brugnano2010,Calvo2010,Gonzalez1996,Hairer2010,Hernandez,Maclachlan1998}).
However, these general-purpose methods are not suitable for dealing
with \eqref{LNODE} when $||QM||$ is very large. Firstly, numerical
solutions generated by them are far from accurate. They are
generally implicit and iterative solutions are required at each
step. But the fixed-point iterations for them are not convergent
unless the stepsize is tiny enough. As mentioned at the beginning,
the two obstacles are hopeful of overcoming by introducing
exponential integrators. In \cite{Wang2012}, the authors proposed an
energy-preserving AAVF integrator (a Trigonometric method) dealing
with the second-order Hamiltonian system:
\begin{equation*}
\left\{\begin{aligned}
&\ddot{q}(t)+\tilde{M}q(t)=\nabla V(q(t)),\quad\text{$\tilde{M}$ is a symmetric matrix},\\
&q(t_{0})=q_{0},\quad\dot{q}(t_{0})=\dot{q}_{0},\\
\end{aligned}\right.
\end{equation*}
which falls into the class of \eqref{LNODE} by introducing
$\dot{q}=p$. In this paper, we present and analyse a new exponential
integrator for \eqref{LNODE} which can preserve the first integral
or the Lyapunov function.

The plan of this paper is as follows. In Section \ref{CCDs}, we
construct a general structure-preserving scheme for \eqref{LNODE}.
In Section \ref{CLA}, we discuss two important properties of the
scheme. Then we present a list of problems which can be solved by
this scheme in Section \ref{LECNS}. Numerical results including the
comparison between our new scheme and other structure-preserving
schemes in the literature are shown in section
\ref{NUMEXPs}. The last Section is concerned with the conclusion.

\section{Construction of the structure-preserving scheme for conservative and dissipative
systems}\label{CCDs}
\begin{preli}
Throughout  this paper, {given a holomorphic function $f$ in the
neighborhood of zero ($f(0):=\lim\limits_{z\to 0} f(z)$ if ~$0$ is a
removable singularity )}:
\begin{equation*}
f(z)=\sum_{i=0}^{\infty}\frac{f^{(i)}(0)}{i!}z^{i},
\end{equation*}
and a matrix $A$, the matrix-valued function $f(A)$ is defined by :
\begin{equation*}
f(A)=\sum_{i=0}^{\infty}\frac{f^{(i)}(0)}{i!}A^{i}.
\end{equation*}
$I$ and $O$ always denote identity and zero matrices of appropriate
dimensions respectively. $A^{\frac{1}{2}}$ is a square root (not
necessarily principal) of a symmetric matrix $A$. {If $f^{(i)}(0)=0$
for odd $i$, then $f(A^{\frac{1}{2}})$ is well-defined for every
symmetric $A$ (independent of the choice of $A^{\frac{1}{2}}$).}
Readers are referred to \cite{Higham} for details about functions of
matrices.
\end{preli}

It is well known that the discrete gradient (DG) method is a popular
tool for constructing energy-preserving schemes. Broadly speaking,
$\overline{\nabla}H(y,\hat{y})$ is said to be a discrete gradient of
function $H$ if
\begin{equation}\label{COND}
\left\{\begin{aligned}
&\overline{\nabla}H(y,\hat{y})^{\intercal}(y-\hat{y})=H(y)-H(\hat{y}),\\
&\overline{\nabla }H(y,y)=\nabla H(y).\\
\end{aligned}\right.
\end{equation}
Accordingly,
\begin{equation}\label{DG}
y^{1}=y^{0}+hJ^{-1}\overline{\nabla}H(y^{1},y^{0})
\end{equation}
is called a DG method for the system \eqref{ODE2}. Multiplying
$\overline{\nabla }H(y^{1},y^{0})^{\intercal}$ on both sides of
\eqref{DG} and using the first identity of \eqref{COND}, we obtain
$H(y^{1})=H(y^{0})$, i.e. the scheme \eqref{DG} is
energy-preserving. For more details on the DG method, readers are
referred to \cite{Gonzalez1996,Maclachlan1999}.

On the other hand, most of exponential integrators can be derived
from the variation-of-constants formula for the problem
\eqref{LNODE}:
\begin{equation}\label{VF}
y(t_{0}+h)=\exp(hQM)y(t_{0})+h\int_{0}^{1}\exp((1-\xi)hQM)Q\nabla U(y(t_{0}+\xi h))d\xi.
\end{equation}
Replacing $\nabla U(y(t_{0}+\xi
h))$ with the discrete gradient $\overline{\nabla}U(y^{1},y^{0})$,
the integral in \eqref{VF} can be approximated by :
\begin{equation*}
\begin{aligned}
&\int_{0}^{1}\exp((1-\xi)hQM)Q\nabla U(y(t_{0}+\xi h))d\xi\\
&\approx(\int_{0}^{1}\exp((1-\xi)hQM)d\xi)Q\overline{\nabla}U(y^{1},y^{0})=\varphi(hQM)Q\overline{\nabla}U(y^{1},y^{0}),\\
\end{aligned}
\end{equation*}
where the scalar function is given by
$$\varphi(z)=(\exp(z)-1)/z.$$ Then we obtain the new scheme:
\begin{equation}\label{AAVF}
y^{1}=\exp(V)y^{0}+h\varphi(V)Q\overline{\nabla}U(y^{1},y^{0}),
\end{equation}
where $V=hQM$ and $y^{1}\approx y(t_{0}+h)$.

Due to the energy-preserving property of the DG method, we are hopeful of
preserving the first integral by \eqref{AAVF} when $Q$ is skew. For
convenience, we denote $\overline{\nabla}U(y^{1},y^{0})$ by
$\overline{\nabla}U$ sometimes. To begin with, we give
the following preliminary lemma.
\begin{lem}\label{lem2.1}
For any symmetric matrix $M$ and scalar $h\geq0$,  the matrix
$$B=\exp(hQM)^{\intercal}M\exp(hQM)-M$$ satisfies:
\begin{equation*}
B=\left\{\begin{aligned} &=0 , \quad \text{if $Q$ is skew-symmetric,}\\
&\leq 0, \quad \text{if $Q \leq 0$.}
\end{aligned}\right.
\end{equation*}
\end{lem}

\begin{proof}
Consider the linear ODE:
\begin{equation}\label{LP}
\dot{y}(t)=QMy(t).
\end{equation}
When $Q$ is skew, \eqref{LP} is a conservative equation with the
first integral $\frac{1}{2}y^{\intercal}My$, and its exact solution
starting from the initial value {$y(0)=y^{0}$} is
$y(t)=\exp(tQM)y^{0}$. From
$\frac{1}{2}y(h)^{\intercal}My(h)=\frac{1}{2}y^{0\intercal}My^{0}$,
we have
\begin{equation*}
\
\frac{1}{2}y^{0\intercal}\exp(hQM)^{\intercal}M\exp(hQM)y^{0}=\frac{1}{2}y^{0\intercal}My^{0}
\end{equation*}
for any vector $y^{0}$. Therefore,
$B=\exp(hQM)^{\intercal}M\exp(hQM)-M$ is skew-symmetric. Since it is
also symmetric, $B=0$. The case that $Q\leq0$ can be proved in a
similar way.
\end{proof}

\begin{mytheo}\label{Ham}
If $Q$ is skew-symmetric, then the scheme \eqref{AAVF} preserves the first integral $H$ in \eqref{LNODE} exactly :
\begin{equation*}
H(y^{0})=H(y^{1}),
\end{equation*}
where $H(y)=\frac{1}{2}y^{\intercal}My+U(y)$.
\end{mytheo}

\begin{proof}
Here we firstly assume that the matrix $M$ is not singular. We next
calculate $\frac{1}{2}y^{1\intercal}My^{1}$. Let
$M^{-1}\overline{\nabla}U=\widetilde{\nabla}U$. Replacing  $y^{1}$
by $\exp(V)y^{0}+h\varphi(V)Q\overline{\nabla}U(y^{1},y^{0})$ leads
to
\begin{equation}\label{I}
\begin{aligned}
&\frac{1}{2}y^{1\intercal}My^{1}\\
&=\frac{1}{2}(y^{0\intercal}\exp(V)^{\intercal}+h\overline{\nabla}U^{\intercal}Q^{\intercal}\varphi(V)^{\intercal})M(\exp(V)y^{0}+h\varphi(V)Q\overline{\nabla}U)\\
&=\frac{1}{2}y^{0\intercal}\exp(V)^{\intercal}M\exp(V)y^{0}+hy^{0\intercal}\exp(V)^{\intercal}M\varphi(V)Q\overline{\nabla}U+
\frac{h^{2}}{2}\overline{\nabla}U^{\intercal}Q^{\intercal}\varphi(V)^{\intercal}M\varphi(V)Q\overline{\nabla}U\\
&=\frac{1}{2}y^{0\intercal}\exp(V)^{\intercal}M\exp(V)y^{0}+y^{0\intercal}\exp(V)^{\intercal}M\varphi(V)V\widetilde{\nabla}U+
\frac{1}{2}\widetilde{\nabla}U^{\intercal}V^{\intercal}\varphi(V)^{\intercal}M\varphi(V)V\widetilde{\nabla}U\quad(\text{using }V=hQM)\\
&=\frac{1}{2}y^{0\intercal}\exp(V)^{\intercal}M\exp(V)y^{0}+y^{0\intercal}\exp(V)^{\intercal}M(\exp(V)-I)\widetilde{\nabla}U\\
&+\frac{1}{2}\widetilde{\nabla}U^{\intercal}(\exp(V)^{\intercal}-I)M(\exp(V)-I)\widetilde{\nabla}U\quad(\text{using }\varphi(V)V=\exp(V)-I)\\
&=\frac{1}{2}y^{0\intercal}\exp(V)^{\intercal}M\exp(V)y^{0}+y^{0\intercal}(\exp(V)^{\intercal}M\exp(V)-\exp(V)^{\intercal}M)\widetilde{\nabla}U\\
&+\frac{1}{2}\widetilde{\nabla}U^{\intercal}(\exp(V)^{\intercal}M\exp(V)-\exp(V)^{\intercal}M-M\exp(V)+M)\widetilde{\nabla}U.\\
\end{aligned}
\end{equation}
On the other hand, it follows from the property of the discrete
gradient \eqref{COND} that
\begin{equation}\label{II}
\begin{aligned}
&U(y^{1})-U(y^{0})\\
&=(y^{1\intercal}-y^{0\intercal})\overline{\nabla}U(y^{1},y^{0})\\
&=y^{0\intercal}(\exp(V)^{\intercal}-I)\overline{\nabla}U
+h\overline{\nabla}U^{\intercal}Q^{\intercal}\varphi(V)^{\intercal}\overline{\nabla}U\\
&=y^{0\intercal}(\exp(V)^{\intercal}M-M)\widetilde{\nabla}U
+\widetilde{\nabla}U^{\intercal}V^{\intercal}\varphi(V)^{\intercal}M\widetilde{\nabla}U\\
&=y^{0\intercal}(\exp(V)^{\intercal}M-M)\widetilde{\nabla}U
+\widetilde{\nabla}U^{\intercal}(\exp(V)^{\intercal}M-M)\widetilde{\nabla}U.\\
\end{aligned}
\end{equation}

Combining \eqref{I}, \eqref{II} and collecting  terms by types
`$y^{0\intercal}*y^{0}$', `$y^{0\intercal}*\widetilde{\nabla}U$',
`$\widetilde{\nabla}U^{\intercal}*\widetilde{\nabla}U$' leads to
\begin{equation}\label{III}
\begin{aligned}
&H(y^{1})-H(y^{0})\\
&=\frac{1}{2}y^{1\intercal}My^{1}-\frac{1}{2}y^{0\intercal}My^{0}+U(y^{1})-U(y^{0})\\
&=\frac{1}{2}y^{0\intercal}(\exp(V)^{\intercal}M\exp(V)-M)y^{0}+y^{0\intercal}(\exp(V)^{\intercal}M\exp(V)-M)\widetilde{\nabla}U\\
&+\frac{1}{2}\widetilde{\nabla}U^{\intercal}(\exp(V)^{\intercal}M\exp(V)-M)\widetilde{\nabla}U
+\frac{1}{2}\widetilde{\nabla}U^{\intercal}(\exp(V)^{\intercal}M-M\exp(V))\widetilde{\nabla}U\\
&=\frac{1}{2}(y^{0}+\widetilde{\nabla}U)^{\intercal}B(y^{0}+\widetilde{\nabla}U)+\frac{1}{2}\widetilde{\nabla}U^{\intercal}C\widetilde{\nabla}U=0,\\
\end{aligned}
\end{equation}
where $B=\exp(V)^{\intercal}M\exp(V)-M$ and
$C=\exp(V)^{\intercal}M-M\exp(V)$. The last step is from the
skew-symmetry of the matrix $B$ (according to Lemma \ref{lem2.1})
and $C$.

If $M$ is singular, it is easy to find a series of symmetric and
nonsingular matrices $\{M_{\varepsilon}\}$ which converge to $M$
when $\varepsilon\to 0.$ Thus, according to the result
stated above, it still holds that
\begin{equation}\label{perturbed}
H_{\varepsilon}(y_{\varepsilon}^{1})=H_{\varepsilon}(y^{0})
\end{equation} for all $\varepsilon,$  where
$H_{\varepsilon}(y)=\frac{1}{2}y^{\intercal}M_{\varepsilon}y+U(y)$
is the first integral of {the perturbed problem}
$$\dot{y}=Q(M_{\varepsilon}y+\nabla U(y)),\quad {y(t_{0})=y^{0}},$$
and
\begin{equation*}
y_{\varepsilon}^{1}=\exp(V_{\varepsilon})y^{0}+h\varphi(V_{\varepsilon})Q\overline{\nabla}U(y_{\varepsilon}^{1},y^{0}),\quad V_{\varepsilon}=hQM_{\varepsilon}.
\end{equation*}
Therefore, when $\varepsilon\to0$, $y^{1}_{\varepsilon}\to y^{1}$ and \eqref{perturbed} leads to
$$H(y^{1})=H(y^{0}).$$ This completes the proof.
\end{proof}\qed

Moreover, the scheme \eqref{AAVF} can also model the decay of the Lyapunov
function once $Q\leq0$ in \eqref{LNODE}. The next theorem shows this point.

\begin{mytheo}\label{Lya}
If $Q$ is negative semi-definite (not necessary to be symmetric),
then the scheme \eqref{AAVF} preserves the Lyapunov function $H$ in
\eqref{LNODE}:
\begin{equation*}
H(y^{1})\leq H(y^{0}),
\end{equation*}
where $H(y)=\frac{1}{2}y^{\intercal}My+U(y)$.
\end{mytheo}
\begin{proof} If $M$ is nonsingular, the equation in \eqref{III}
\begin{equation*}
H(y^{1})-H(y^{0})=\frac{1}{2}(y^{0}+\widetilde{\nabla}U)^{\intercal}B(y^{0}+\widetilde{\nabla}U)
\end{equation*}
still holds, since the derivation does not depend on the
skew-symmetry of $Q$. By Lemma \ref{lem2.1}, $B$ is negative
semi-definite. Thus $H(y^{1})\leq H(y^{0}).$ In the case that $M$ is
singular, this theorem can be easily proved by replacing the
equalities
$$H_{\varepsilon}(y_{\varepsilon}^{1})=H_{\varepsilon}(y^{0}),\quad H(y^{1})=H(y^{0})$$ in
the proof of Theorem \ref{Ham} with the inequalities
$$H_{\varepsilon}(y_{\varepsilon}^{1})\leq H_{\varepsilon}(y^{0}),\quad H(y^{1})\leq H(y^{0}).$$
\end{proof} We here skip the details.\qed

In this paper, we choose the average vector field (AVF) as the
discrete gradient in \eqref{AAVF}. The corresponding scheme for
\eqref{LNODE} now is {\begin{equation}\label{EAVF}
\begin{aligned}
y^{1}&=\exp(V)y^{0}+h\varphi(V)Q\int_{0}^{1}\nabla
U((1-\tau)y^{0}+\tau y^{1})d\tau,\end{aligned}
\end{equation}
}where $V=hQM$ and $y^{1}\approx y(t_{0}+h)$. \eqref{EAVF} is
{called an \emph{exponential AVF integrator} and denoted by EAVF.}

\section{Properties of EAVF}\label{CLA}
In this section, we present two properties of EAVF as
follows.
\begin{mytheo}\label{symm}
The EAVF integrator \eqref{EAVF} is symmetric.
\end{mytheo}

\begin{proof}
Exchanging $y^{0}\leftrightarrow y^{1}$ and replacing $h$ by $-h$ in
\eqref{EAVF}, we obtain
\begin{equation}\label{X}
y^{0}=\exp(-V)y^{1}-h\varphi(-V)Q\int_{0}^{1}\nabla U((1-\tau)y^{1}+\tau y^{0})d\tau .
\end{equation}
Rewrite \eqref{X} as:
\begin{equation}\label{XX}
y^{1}=\exp(V)y^{0}+h\exp(V)\varphi(-V)Q\int_{0}^{1}\nabla U((1-\tau)y^{0}+\tau y^{1})d\tau .
\end{equation}
Since $\exp(V)\varphi(-V)=\varphi(V)$, \eqref{XX} is the same as
\eqref{EAVF} exactly, which means that EAVF is symmetric.
\end{proof}\qed

It should be noted that the scheme \eqref{EAVF} is implicit in
general, and thus iteration solutions are required. Next, we discuss
the convergence of the fixed-point iteration for the EAVF
integrator.

\begin{mytheo}\label{iter}
Suppose that $||\varphi(V)||_{2}\leq C$, $\nabla U(u)$ satisfies the Lipschitz condition, i. e.
there exists a constant L such that
$$||\nabla U(v)-\nabla U(w)||_{2}\leq L||v-w||_{2}.$$
If
\begin{equation}\label{cond}
0<h\leq\hat h<\frac{2}{CL||Q||_{2}}
\end{equation}
then the iteration $$\Psi :
z\mapsto\exp(V)y^{n}+h\varphi(V)Q\int_{0}^{1}\nabla
U((1-\tau)y^{n}+\tau z)d\tau$$ for the EAVF integrator \eqref{EAVF}
is convergent.
\end{mytheo}

\begin{proof} Since
\begin{equation*}
\begin{aligned}
&||\Psi(z_{1})-\Psi(z_{2})||_{2}\\
&=||h\varphi(V)Q\int_{0}^{1}(\nabla U((1-\tau)y^{n}+\tau z_{1})-\nabla U((1-\tau)y^{n}+\tau z_{2}))d\tau||_{2}\\
&\leq \frac{h}{2}CL||Q||_{2}||z_{1}-z_{2}||_{2}\leq\frac{\hat h}{2}CL||Q||_{2}||z_{1}-z_{2}||_{2}=\rho||z_{1}-z_{2}||_{2},\\
\end{aligned}
\end{equation*}
where $$\rho=\frac{\hat h}{2}CL||Q||_{2}<1,$$ the iteration $\Psi$
converges by \eqref{cond} and Contraction Mapping Theorem.
\end{proof}\qed
\begin{rem}\label{exam}
We note two special and important cases in practical applications.
If $QM$ is skew-symmetric or symmetric negative semi-definite, then
the spectrum of $V$ lies in the left {half-plane.} Since $QM$ is
unitarily diagonalizable and $|\varphi(z)|\leq1$ for any {$z$
satisfying $real(z)\leq0$, we have} $||\varphi(V)||_{2}\leq1$.
\end{rem}

In many cases, the matrix $M$ has extremely large norm (e.g., $M$ incorporates high frequency components in
oscillatory problems or $M$ is the differential matrix in semi-discrete PDEs), thus,
Theorem \ref{iter} ensures the possibility of choosing relatively
large stepsize regardless of $M$.

In practice, the integral in \eqref{EAVF} usually cannot be easily
calculated. Therefore, we can evaluate it using the $s$-point
Gauss-Legendre (GL$s$) formula $(b_{i},c_{i})_{i=1}^{s}$:
\begin{equation*}
\int_{0}^{1}\nabla U((1-\tau)y^{0}+\tau
y^{1})d\tau\approx\sum_{i=1}^{s}b_{i}\nabla
U((1-c_{i})y^{0}+c_{i}y^{1})).
\end{equation*}
The corresponding scheme is denoted by EAVFGL$s$. Since the
$s$-point GL quadrature formula is symmetric, EAVFGL$s$ is also
symmetric. According to $\sum_{i=1}^{s}b_{i}c_{i}=1/2$, the
corresponding iteration for EAVFGL$s$ is convergent
provided \eqref{cond} holds.

\section{Problems suitable for the EAVF}\label{LECNS}
\subsection{Highly oscillatory nonseparable Hamiltonian system}
Consider the Hamiltonian
$$H(p,q)=\frac{1}{2}p_{1}^{\intercal}M_{1}^{-1}p_{1}+\frac{1}{2\varepsilon^{2}}q_{1}^{\intercal}A_{1}q_{1}+S(p,q),$$
where $$p=\left(\begin{array}{c}p_{0}\\p_{1}\end{array}\right),\quad
q=\left(\begin{array}{c}q_{0}\\q_{1}\end{array}\right)$$ {are both
$d$-length vectors,} $M_{1}, A_{1}$ are symmetric positive definite
matrices, and $0<\varepsilon\ll 1$. This Hamiltonian governs
oscillatory mechanical systems in 2 or 3 space dimensions such as
the stiff spring pendulum and the dynamics of the multi-atomic
molecule (see, e.g. \cite{Cohen2006,Cohen2006b}). After an
appropriate canonical transformation (see, e.g. \cite{Hairer}), this
Hamiltonian becomes:
\begin{equation}\label{Hamil}
H(p,q)=\frac{1}{2}\sum_{j=1}^{l}(p_{1,j}^{2}+\frac{\lambda_{j}^{2}}{\varepsilon^{2}}q_{1,j}^{2})+S(p,q),
\end{equation}
where $p_{1}=(p_{1,1},\ldots,p_{1,l})^{\intercal},
q_{1}=(q_{1,1},\ldots,q_{1,l})^{\intercal}$. The corresponding
equation is given by
\begin{equation}
\left\{\begin{aligned}\label{NSP}
&\dot{p_{0}}=-\nabla_{q_{0}}S(p,q),\\
&\dot{p_{1}}=-\omega^{2\intercal}q_{1}-\nabla_{q_{1}}S(p,q),\\
&\dot{q_{0}}=p_{0}+(\nabla_{p_{0}}S(p,q)-p_{0}),\\
&\dot{q_{1}}=p_{1}+\nabla_{p_{1}}S(p,q),\\
\end{aligned}\right.
\end{equation}
where $\omega=(\omega_{1},\ldots,\omega_{l})^{\intercal},
\omega_{j}=\lambda_{j}/\varepsilon$ for $j=1,\ldots,l$. \eqref{NSP}
is of the form \eqref{LNODE} : \begin{equation*}
y=\left(\begin{array}{c}p\\q\end{array}\right),\quad
Q=\left(\begin{array}{cc}O&-I_{d\times d}\\I_{d\times
d}&O\end{array}\right),\quad M=\left(\begin{array}{cc}I_{d\times
d}&O\\ O&\Omega_{d\times d}\end{array}\right),
\end{equation*}
and
\begin{equation*}
U(p,q)=S(p,q)-\frac{1}{2}p_{0}^{\intercal}p_{0},\quad\Omega=diag(0,\ldots,0,\omega_{1}^{2},\ldots,\omega_{l}^{2}).
\end{equation*}
Since $q_{11},\ldots,q_{1l}$ and $p_{11},\ldots,p_{1l}$ are fast variables, it is favorable to integrate the linear part
of them exactly by the scheme \eqref{EAVF}. Note that
\begin{equation*}
\varphi(V)=\left(\begin{array}{cc}sinc(h\Omega^{\frac{1}{2}})&h^{-1}g_{2}(h\Omega^{\frac{1}{2}})\\hg_{1}(h\Omega^{\frac{1}{2}})&sinc(h\Omega^{\frac{1}{2}})\end{array}\right),
\end{equation*}
where $sinc(z)=\sin(z)/z, g_{1}(z)=(1-\cos(z))/z^{2},
g_{2}(z)=\cos(z)-1$. {Unfortunately, the block
$h^{-1}g_{2}(h\Omega^{\frac{1}{2}})$ is not uniformly bounded. In
the first experiment, the iteration still works well, perhaps due to
the small Lipshitz constant of $\nabla S$.}

\subsection{Second-order (damped) highly oscillatory system}
Consider
\begin{equation}\label{2th}
\ddot{q}-N\dot{q}+\Omega q=-\nabla U_{1}(q),
\end{equation}
where {$q$ is a $d$-length vector variable, $U_{1}:
\mathbb{R}^{d}\to\mathbb{R}$ is a differential function}, $N$ is a
symmetric negative semi-definite matrix, $\Omega$ is a symmetric
positive semi-definite matrix, $||\Omega||$ or $||N||\gg1$.
\eqref{2th} stands for highly oscillatory problems such as the
dissipative molecular dynamics, the (damped) Duffing and
semi-discrete nonlinear wave equations. By introducing $p=\dot{q},$
we write \eqref{2th} as a first-order system of ODEs :
\begin{equation}\label{1th}
\left(\begin{array}{c}\dot{q}\\ \dot{p}\end{array}\right)=\left(\begin{array}{cc}O&I\\-\Omega&N\end{array}\right)
\left(\begin{array}{c}q\\p\end{array}\right)+\left(\begin{array}{c}0\\-\nabla U_{1}(q)\end{array}\right),
\end{equation}
which falls into the {class \eqref{LNODE},} where
\begin{equation*}
y=\left(\begin{array}{c}q\\p\end{array}\right), Q=\left(\begin{array}{cc}O&I\\-I&N\end{array}\right), M=\left(\begin{array}{cc}\Omega&O\\O&I\end{array}\right),
U(y)=\left(\begin{array}{c}U_{1}(q)\\O\end{array}\right).
\end{equation*}

{Clearly,} $Q\leq0$ and \eqref{1th} is a dissipative system with the
Lyapunov function
$H=\frac{1}{2}p^{\intercal}p+\frac{1}{2}q^{\intercal}\Omega
q+U_{1}(q)$. In the particular case $N=0,$ \eqref{1th} becomes a
conservative Hamiltonian system. {Let
$$A=QM=\left(\begin{array}{cc}O&I\\-\Omega&N\end{array}\right).$$}
Applying the EAVF integrator \eqref{EAVF} to the equation
\eqref{1th} yields the scheme:
\begin{equation}\label{EAVF1}
\left\{\begin{aligned}
&q^{1}=\exp_{11}q^{0}+\exp_{12}p^{0}-h\varphi_{12}\int_{0}^{1}\nabla U_{1}((1-\tau)q^{0}+\tau q^{1})d\tau,\\
&p^{1}=\exp_{21}q^{0}+\exp_{22}p^{0}-h\varphi_{22}\int_{0}^{1}\nabla U_{1}((1-\tau)q^{0}+\tau q^{1})d\tau,\\
\end{aligned}\right.
\end{equation}
where $\exp(hA)$ and $\varphi(hA)$ are partitioned into
\begin{equation*}
\left(\begin{array}{cc}\exp_{11}&\exp_{12}\\ \exp_{21}&\exp_{22}\end{array}\right) \text{  and  } \left(\begin{array}{cc}\varphi_{11}&\varphi_{12}\\ \varphi_{21}&\varphi_{22}\end{array}\right),
\end{equation*}
respectively.

It should be noted that only the first equation in the scheme
\eqref{EAVF1} need to be solved by iterations. From the proof
procedure of Theorem \ref{iter}, one can find that the convergence
of the fixed-point iteration for \eqref{EAVF1} is irrelevant to
$||A||$ provided $\varphi_{12}$ is uniformly bounded.
\begin{mytheo}\label{iter2}
{Assume that $\Omega$ commutes with $N$, $||\nabla U_{1}(v)-\nabla
U_{1}(w)||_{2}\leq L||v-w||_{2}$, then the iteration
$$\Phi : z\mapsto\exp_{11}q^{0}+\exp_{12}p^{0}-h\varphi_{12}\int_{0}^{1}\nabla U_{1}((1-\tau)q^{0}+\tau z)d\tau$$
for the scheme \eqref{EAVF1} is convergent provided
$$0<h\leq\hat h<\frac{2}{L^{\frac{1}{2}}}.$$}
\end{mytheo}
\begin{proof} The crucial point here is to find a uniform upper bound of $||\varphi_{12}||$.
Since $\Omega$ commutes with $N$, they can be simultaneously diagonalized:
$$\Omega=F^{\intercal}\Lambda F,\quad N=F^{\intercal}\Sigma F,$$
where $F$ is an orthogonal matrix,
$\Lambda=diag(\lambda_{1},\ldots,\lambda_{d}),
\Sigma=diag(\sigma_{1},\ldots,\sigma_{d})$ and $\lambda_{i}\geq0,
\sigma_{i}\leq0$ for $i=1,2,\ldots,d$. It now
follows from
\begin{equation*}
\begin{aligned}
&A=\left(\begin{array}{cc}F^{\intercal}&O\\O&F^{\intercal}\end{array}\right)\left(\begin{array}{cc}O&I\\-\Lambda&\Sigma\end{array}\right)
\left(\begin{array}{cc}F&O\\O&F\end{array}\right)\end{aligned}\end{equation*}
that
\begin{equation*}
\begin{aligned}
&\exp(hA)=\left(\begin{array}{cc}F^{\intercal}&O\\O&F^{\intercal}\end{array}\right)\exp\left\{\left(\begin{array}{cc}O&hI\\-h\Lambda&h\Sigma\end{array}\right)\right\}
\left(\begin{array}{cc}F&O\\O&F\end{array}\right).\\
\end{aligned}
\end{equation*}
To show that $\exp_{12}$ and $\varphi_{12}$ depends on $h$, we
denote them by $\exp_{12}^{h}$ and $\varphi_{12}^{h}$, respectively.
After some calculations, we have
\begin{equation*}
\exp_{12}^{h}=F^{\intercal}\frac{2\sinh(h(\Sigma^{2}-4\Lambda)^{\frac{1}{2}}/2)}{(\Sigma^{2}-4\Lambda)^{\frac{1}{2}}}\exp\left(\frac{h\Sigma}{2}\right)F.
\end{equation*}
 Then we have
\begin{equation}\label{exp}
||\exp_{12}^{h}||_{2}=||\frac{2\sinh(h(\Sigma^{2}-4\Lambda)^{\frac{1}{2}}/2)}{(\Sigma^{2}-4\Lambda)^{\frac{1}{2}}}\exp\left(\frac{h\Sigma}{2}\right)||_{2}
=h\max_{i}|\frac{\sinh((h^{2}\sigma_{i}^{2}/4-\lambda_{i})^{\frac{1}{2}})}{(h^{2}\sigma_{i}^{2}/4-\lambda_{i})^{\frac{1}{2}}}\exp\left(\frac{h\sigma_{i}}{2}\right)|.
\end{equation}
In order to estimate $||\exp_{12}^{h}||_{2}$, the bound of the
function
\begin{equation*}
g(\lambda,\sigma)=\frac{\sinh((\sigma^{2}-4\lambda)^{\frac{1}{2}})}{(\sigma^{2}-4\lambda)^{\frac{1}{2}}}\exp\left(\sigma\right),
\end{equation*}
should be considered for $\sigma\leq0,\lambda\geq0$. If
$\sigma^{2}-4\lambda<0$, we set
$(\sigma^{2}-4\lambda)^{\frac{1}{2}}=ia,$ where $i$ is the imaginary
unit and $a$ is a real number. Then we have
\begin{equation*}
|g|=|\frac{\sin(a)}{a}\exp\left(\sigma\right)|\leq|\frac{\sin(a)}{a}|\leq1.
\end{equation*}
If $\sigma^{2}-4\lambda\geq0$, then $a=(\sigma^{2}-4\lambda)^{\frac{1}{2}}\leq-\sigma$,
\begin{equation*}
|g|=|\frac{\sinh(a)}{a}\exp\left(\sigma\right)|\leq
|\frac{\sinh(a)}{a}\exp(-a)|=|\frac{1-\exp(-2a)}{2a}|\leq1.
\end{equation*}
Thus
\begin{equation}\label{g}
|g(\lambda,\sigma)|\leq1\quad\text{for }\sigma\leq0,\lambda\geq0.
\end{equation}
It follows from \eqref{exp} and \eqref{g} that
\begin{equation}\label{bound}
||\exp_{12}^{h}||_{2}=h\max_{i}|g(\frac{h\sigma_{i}}{2},\lambda_{i})|\leq h.
\end{equation}
Therefore, using $\varphi(hA)=\int_{0}^{1}\exp((1-\xi)hA)d\xi$ and
\eqref{bound}, we obtain
$$||\varphi_{12}^{h}||_{2}=||\int_{0}^{1}\exp_{12}^{(1-\xi)h}d\xi||_{2}\leq\int_{0}^{1}||\exp_{12}^{(1-\xi)h}||_{2}d\xi
\leq\int_{0}^{1}(1-\xi)hd\xi=\frac{1}{2}h.$$

\noindent Since the rest of the proof is very similar to that of
Theorem \ref{iter}, we omit it here.
\end{proof}\qed

It can be observed that in the particular case that $N=0$, the
scheme \eqref{EAVF1} reduces to the AAVF integrator in
\cite{Wang2012}.

\subsection{Semi-discrete conservative and dissipative PDEs}
Many time-dependent {PDEs are of the form} :
\begin{equation}\label{TDPDE}
\partial_{t}y(x,t)=\mathcal{Q}\frac{\delta\mathcal{H}}{\delta y},
\end{equation}
where $y(\cdot,t)\in X$ for every $t\geq0$, $X$ is a {Hilbert} space
such as $\mathbf{L}^{2}(\Omega),
\mathbf{L}^{2}(\Omega)\times\mathbf{L}^{2}(\Omega), \ldots$,
$\Omega$ is a domain in $\mathbb{R}^{d}$, and $\mathcal{Q}$ is a
linear operator on $X$.
$\mathcal{H}[y]=\int_{\Omega}H(y,\partial_{\alpha}y)dx$ ($H$ is
smooth, $x=(x_{1},\ldots,x_{d}), dx=dx_{1}\ldots dx_{d}$ and
$\partial_{\alpha}y$ denotes the partial derivatives of $y$ with
respect to spatial variables $x_{i}, 1\leq i\leq d$). Under suitable
boundary condition (BC), the variational derivative
$\frac{\delta\mathcal{H}}{\delta y}$ is defined by :
$$\langle\frac{\delta\mathcal{H}}{\delta y},z\rangle=\frac{d}{d\varepsilon}|_{\varepsilon=0}\mathcal{H}[y+\varepsilon z]$$
for any smooth $z\in X$ satisfying the same BC, where
$\langle\cdot,\cdot\rangle$ is the inner product {of $X$.} If
$\mathcal{Q}$ is a skew or negative semi-definite operator with
respect to $\langle\cdot,\cdot\rangle$, then the equation
\eqref{TDPDE} is conservative (e.g., the nonlinear wave, nonlinear
Schr\"{o}dinger, Korteweg--de Vries and Maxwell equations) or
dissipative (e.g., the Allen--Cahn, Cahn--Hilliard, Ginzburg--Landau
and heat equations), i.e., $\mathcal{H}[y]$ is constant or
{monotonically decreasing} (see, e.g.
\cite{Celledoni2010,Furihata2012}). In general, after the spatial
discretization, \eqref{TDPDE} becomes a conservative or dissipative
system of ODEs in the form \eqref{LNODE}. Here we exemplify
conservative ones by the nonlinear Schr\"{o}dinger equation:
\begin{equation}\label{NLS}
iy_{t}+y_{xx}+V^{'}(|y|^{2})y=0
\end{equation}
under the periodic BC $y(0,t)=y(L,t)$, where $y=p+iq$ is {a
complex-valued function, $p,q$ are both real}, $i$ is the imaginary
unit. The equation \eqref{NLS} {is of the form \eqref{TDPDE} :
\begin{equation}\label{RNLS}
\partial_{t}\left(\begin{array}{c}p\\q\end{array}\right)=\left(\begin{array}{cc}0&-1\\1&0\end{array}\right)
\left(\begin{array}{c}p_{xx}+V^{'}(p^{2}+q^{2})p\\q_{xx}+V^{'}(p^{2}+q^{2})q\end{array}\right),
\end{equation}}
where $X=\mathbf{L}^{2}([0,L])\times\mathbf{L}^{2}([0,L]),
\mathcal{H}[y]=\frac{1}{2}\int_{0}^{L}(V(p^{2}+q^{2})-p_{x}^{2}-q_{x}^{2})dx$.
Assume that the spatial domain is equally partitioned into $N$
intervals: $0=x_{0}<x_{1}<\ldots<x_{N}=L$. Discretizing the spatial
derivatives of \eqref{RNLS} by the central difference arrives at
\begin{equation}\label{DNLS}
\left(\begin{array}{c}\dot{\tilde{p}}\\ \dot{\tilde{q}}\end{array}\right)=\left(\begin{array}{cc}O&-I\\I&O\end{array}\right)
\left(\begin{array}{c}D\tilde{p}+V^{'}(\tilde{p}^{2}+\tilde{q}^{2})\tilde{p}\\D\tilde{q}+V^{'}(\tilde{p}^{2}+\tilde{q}^{2})\tilde{q}\end{array}\right),
\end{equation}
{where
\begin{equation*}
D=\left(\begin{array}{ccccc}-2&1& & & \\1&-2&1& & \\&\ddots&\ddots&\ddots& \\ & &1&-2&1\\& & &1&-2\end{array}\right),
\end{equation*}
is an $N\times N$ symmetric differential matrix,}
$\tilde{p}=(p_{0},\ldots,p_{N-1})^{\intercal},
\tilde{q}=(q_{0},\ldots,q_{N-1})^{\intercal}, p_{i}(t)\approx
p(x_{i},t)$ and $q_{i}(t)\approx q(x_{i},t)$ for $i=0,\ldots,N-1$.

An example of dissipative PDEs is the Allen-Cahn equation:
\begin{equation}\label{AC}
y_{t}=dy_{xx}+y-y^{3},\quad d\geq0
\end{equation}
under the the Neumann BC $y_{x}(0,t)=y_{x}(L,t)$.
$X=\mathbf{L}^{2}([0,L]), \mathcal{Q}=-1,
\mathcal{H}[y]=\int(\frac{1}{2}dy_{x}^{2}-\frac{1}{2}y^{2}+\frac{1}{4}y^{4})dx$.
{The spatial grids are chosen in the same way as NLS.} Discretizing
the spatial derivative by the central difference, we obtain
\begin{equation}\label{DAC}
\dot{\tilde{y}}=d\hat{D}\tilde{y}+\tilde{y}-\tilde{y}^{3},
\end{equation}
{where
\begin{equation*}
\hat{D}=\left(\begin{array}{ccccc}-1&1& & & \\1&-2&1& & \\&\ddots&\ddots&\ddots& \\ & &1&-2&1\\& & &1&-1\end{array}\right),
\end{equation*}
 is the $(N-1)\times(N-1)$ symmetric differential matrix,}
$\tilde{y}=(y_{1},\ldots,y_{N-1})^{\intercal}, y_{i}(t)\approx
y(x_{i},t)$.

Both the semi-discrete NLS equation \eqref{DNLS} and AC equation
\eqref{DAC} are of the form \eqref{LNODE}. For the NLS and  the AC
equations, we have
\begin{equation*}
Q=\left(\begin{array}{cc}O&-I\\I&O\end{array}\right),\quad
M=\left(\begin{array}{cc}D&O\\O&D\end{array}\right),\quad
U=\frac{1}{2}\sum_{i=0}^{N-1}V(p_{i}^{2}+q_{i}^{2}),
\end{equation*}
and
\begin{equation*}
Q=-I,\quad M=-d\hat{D},\quad U=\sum_{i=1}^{N-1}(-\frac{1}{2}y_{i}^{2}+\frac{1}{4}y_{i}^{4}).
\end{equation*}
respectively. Therefore, the scheme \eqref{EAVF} can be applied to
solve them. Since the matrix $QM$ is skew or symmetric negative
semi-definite in these two cases, according to the Remark
\ref{exam}, the convergence of fixed-point iterations for them is
independent of the differential matrix.

\section{Numerical experiments}\label{NUMEXPs}
{In this section, we compare the EAVF method \eqref{EAVF} with the
well-known implicit midpoint method which is denoted by MID:
\begin{equation}\label{MID}
y^{1}=y^{0}+hQ\nabla\widetilde{U}(\frac{y^{0}+y^{1}}{2}),
\end{equation}
and the traditional AVF method for the equation \eqref{LNODE} is
given by
\begin{equation}\label{AVF}
y^{1}=y^{0}+hQ\int_{0}^{1}\nabla\widetilde{U}((1-\tau)y^{0}+\tau y^{1})d\tau,
\end{equation}
where $\widetilde{U}(y)=U(y)+\frac{1}{2}y^{\intercal}My$. {The
authors in  \cite{Maclachlan1999} showed that \eqref{AVF} preserves
the first integral or the Lyapunov function $\widetilde{U}$. Our
comparison also includes another energy-preserving method of  order
four for \eqref{LNODE} :}
\begin{equation}\label{CRK}
\left\{\begin{aligned}
&y^{\frac{1}{2}}=y^{0}+hQ\int_{0}^{1}(\frac{5}{4}-\frac{3}{2}\tau)\nabla\widetilde{U}(y_{\tau})d\tau,\\
&y^{1}=y^{0}+hQ\int_{0}^{1}\nabla\widetilde{U}(y_{\tau})d\tau,\\
\end{aligned}\right.
\end{equation}
where
$$y_{\tau}=(2\tau-1)(\tau-1)y^{0}-4\tau(\tau-1)y^{\frac{1}{2}}+(2\tau-1)\tau
y^{1}.$$ This method is denoted by CRK since it can be written as a
continuous Runge--Kutta method. For details, readers are referred to
\cite{Hairer2010}.

Throughout the experiment, the `reference solution' is computed by
high-order methods with a sufficiently small stepsize.
 We always start to calculate from
$t_{0}=0$. $y^{n}\approx y(t_{n})$ is obtained by the time-stepping
way $y^{0}\rightarrow y^{1} \rightarrow\cdots\rightarrow
y^{n}\rightarrow\cdots$ for $n=1,2,\ldots$ and $t_{n}=nh$. The error
tolerance for iteration solutions of the four methods is set as
$10^{-14}$. The maximum global error (GE) over the total time
interval is defined by:
\begin{equation*}
GE=\max_{n\geq0}||y^{n}-y(t_{n})||_{\infty}.
\end{equation*}
The maximum global error of $H$ ($EH$) on the interval is:
\begin{equation*}
EH=\max_{n\geq0}|H^{n}-H(y(t_{n}))|.
\end{equation*}}
In our numerical experiments, the computational cost of each method
is measured by the number of function evaluations (FE).

\begin{myexp}
The motion of a triatomic molecule can be modeled by a Hamiltonian
system with the Hamiltonian of the form \eqref{Hamil} (see, e.g.
\cite{Cohen2006}):
\begin{equation}\label{realHamil}
H(p,q)=S(p,q)+\frac{1}{2}(p_{1,1}^{2}+p_{1,2}^{2}+p_{1,3}^{2})+\frac{\omega^{2}}{2}(q_{1,1}^{2}+q_{1,2}^{2}+q_{1,3}^{2}),
\end{equation}
where
$$S(p,q)=\frac{1}{2}p_{0}^{2}+\frac{1}{4}(q_{0}-q_{1,3})^{2}-\frac{1}{4}\frac{2q_{1,2}+q_{1,2}^{2}}{(1+q_{1,2})^{2}}(p_{0}-p_{1,3})^{2}
-\frac{1}{4}\frac{2q_{1,1}+q_{1,1}^{2}}{(1+q_{1,1})^{2}}(p_{0}+p_{1,3})^{2}.$$
The initial values are given by :
\begin{equation*}
\left\{\begin{aligned}
&p_{0}(0)=p_{1,1}(0)=p_{1,2}(0)=p_{1,3}(0)=1,\\
&q_{0}(0)=0.4, q_{1,1}(0)=q_{1,2}(0)=\frac{1}{\omega}, q_{1,3}=\frac{1}{2^{\frac{1}{2}}\omega}.\\
\end{aligned}\right.
\end{equation*}
Setting $h=1/2^{i}, i=6,\ldots,10, \omega=50$ and $h=1/100\times1/2^{i}, i=0,\ldots,4, \omega=100$,
we integrate the problem \eqref{NSP} with the Hamiltonian
\eqref{realHamil} over the interval $[0,50]$. Since the nonlinear
term $\nabla S(p,q)$ is complicated to be integrated, we evaluate
the integrals in EAVF, AVF and CRK by the $3$-point Gauss--Legendre (GL)
quadrature formula $(b_{i},c_{i})_{i=1}^{3}$:
$$b_{1}=\frac{5}{18}, b_{2}=\frac{4}{9}, b_{3}=\frac{5}{18};
\quad c_{1}=\frac{1}{2}-\frac{15^{\frac{1}{2}}}{10},
c_{2}=\frac{1}{2}, c_{3}=\frac{1}{2}+\frac{15^{\frac{1}{2}}}{10}.$$
Corresponding schemes are denoted by EAVFGL3, AVFGL3 and CRKGL3
respectively. Numerical results are presented in Figs. \ref{tri}.

Figs. \ref{tri}(a) and \ref{tri}(c) show that MID and AVFGL3 lost
basic accuracy. It can be observed from \ref{tri}(b) and
\ref{tri}(d) that AVFGL3, EAVFGL3, CRKGL3 are much more efficient in
preserving energy than MID. In the aspects of both energy
preservation and algebraic accuracy, EAVF is the most efficient
among the four methods.

\begin{figure}[ptb]
\centering
\begin{tabular}[c]{cccc}%
  \subfigure[]{\includegraphics[width=6cm,height=7cm]{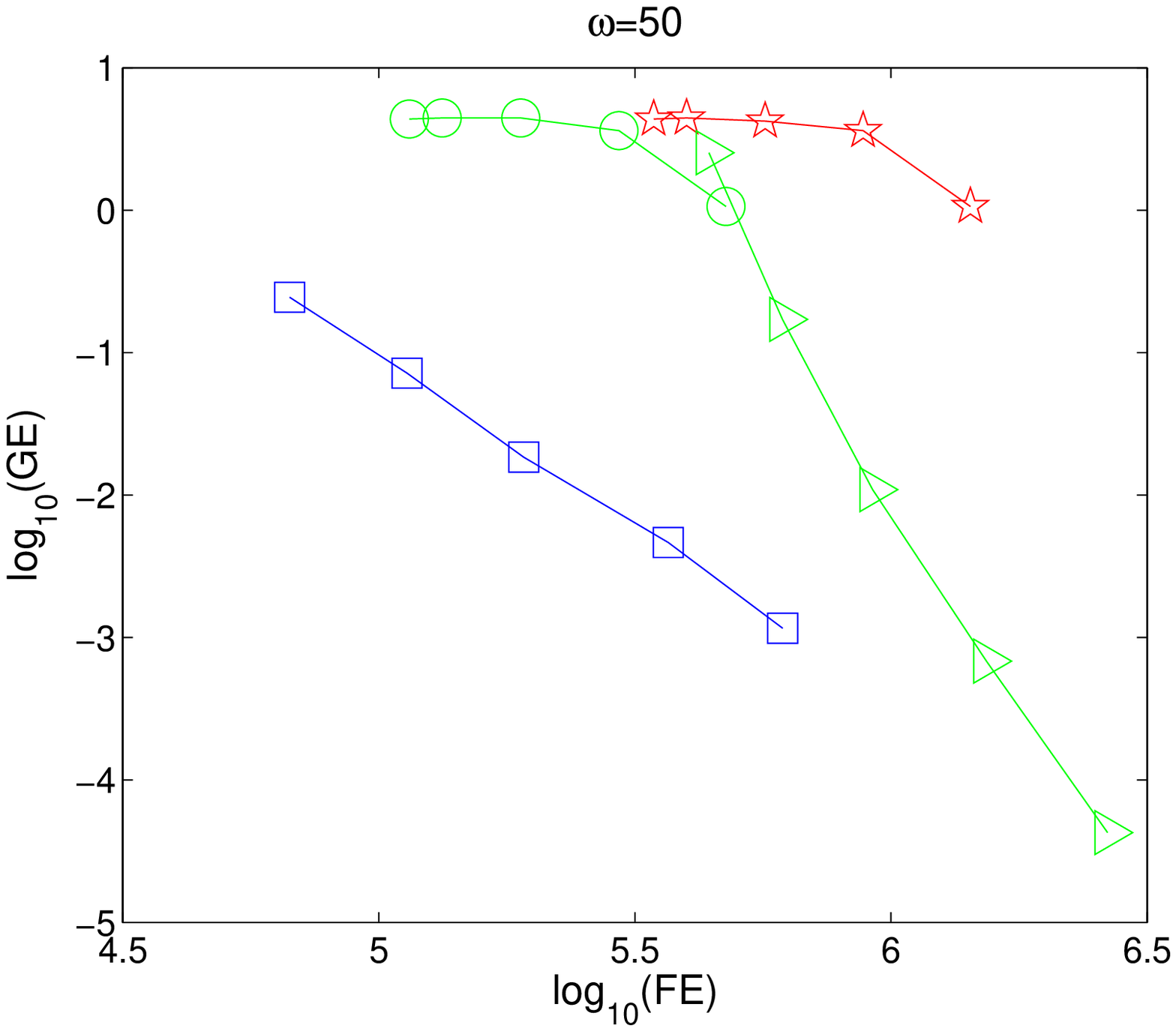}}
  \subfigure[]{\includegraphics[width=6cm,height=7cm]{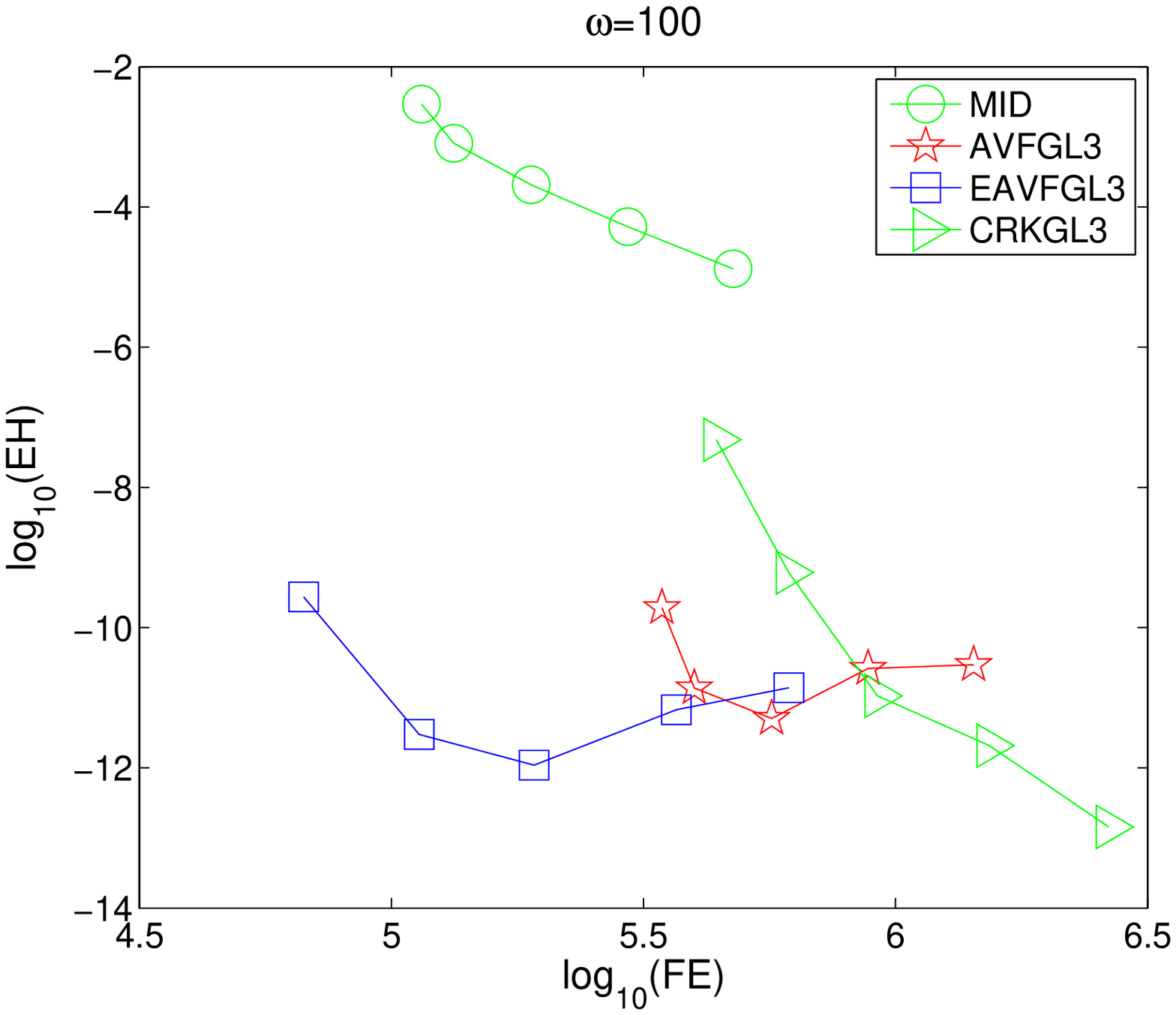}}
\end{tabular}
\begin{tabular}[c]{cccc}%
  \subfigure[]{\includegraphics[width=6cm,height=7cm]{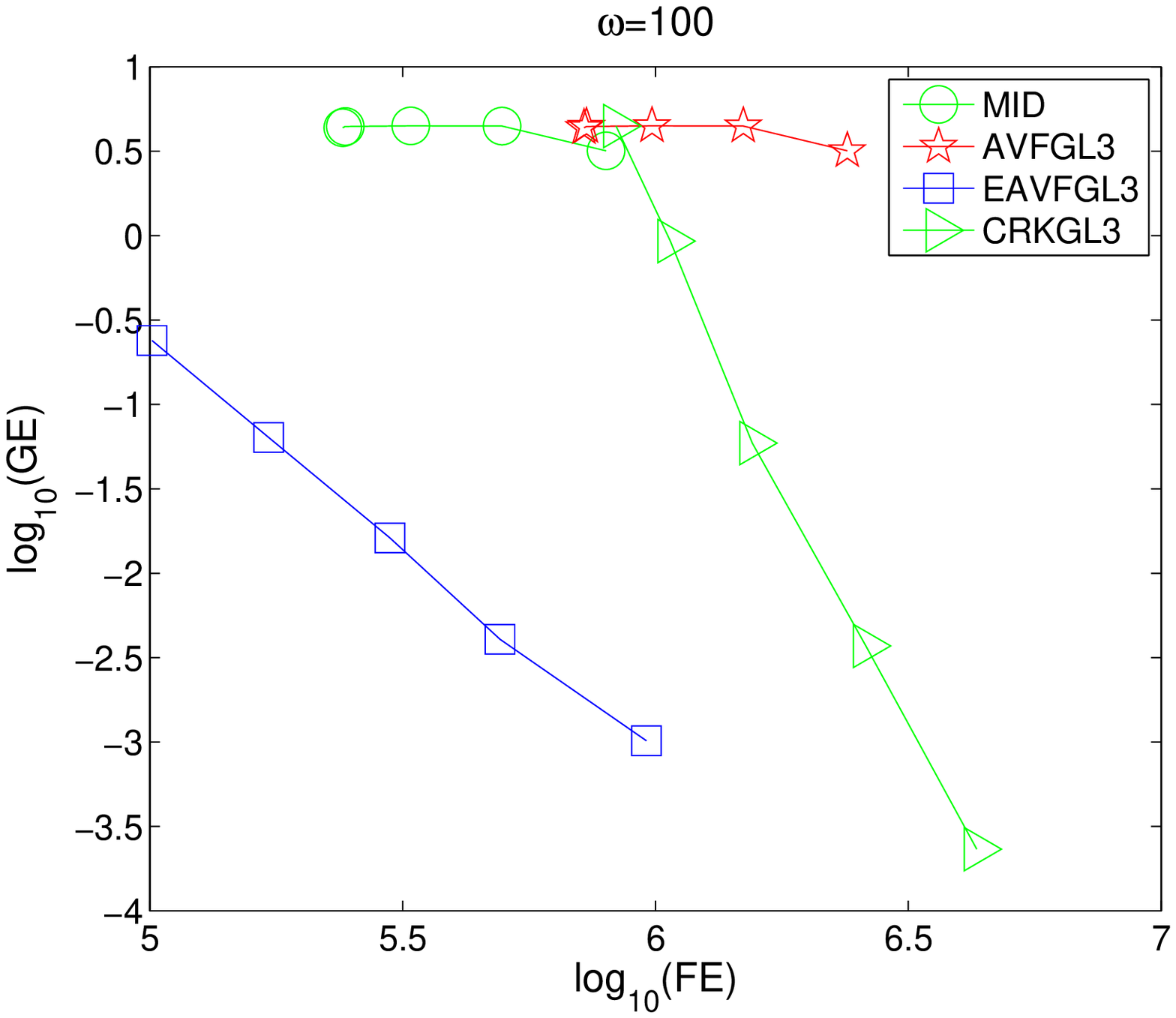}}
  \subfigure[]{\includegraphics[width=6cm,height=7cm]{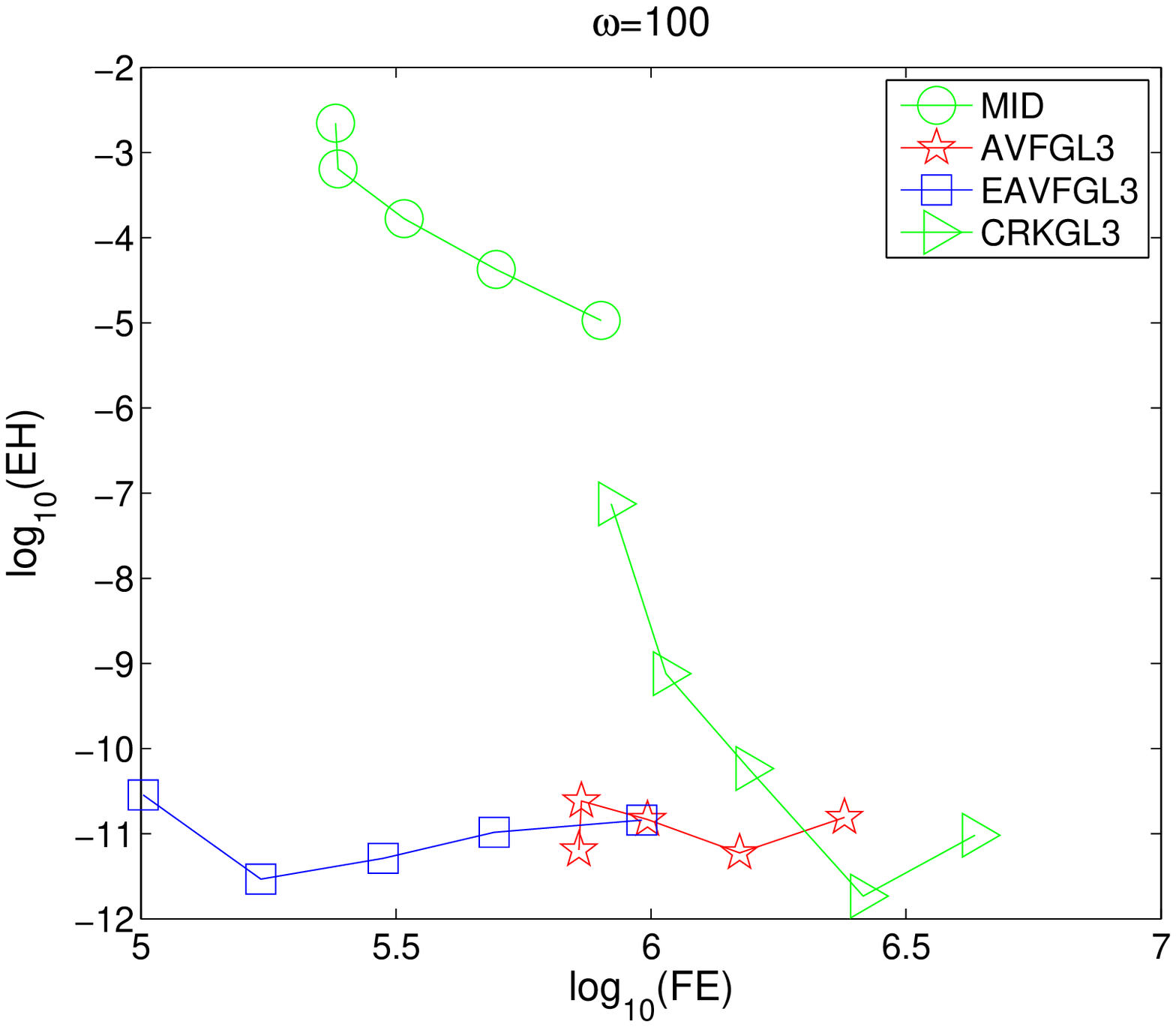}}
\end{tabular}
\caption{Efficiency curves.}
\label{tri}
\end{figure}

\end{myexp}

\begin{myexp}
The equation
\begin{equation}\label{wind}
\begin{aligned}
&\dot{x}_{1}=-\zeta x_{1}-\lambda x_{2}+x_{1}x_{2},\\
&\dot{x}_{2}=\lambda x_{1}-\zeta x_{2}+\frac{1}{2}(x_{1}^{2}-x_{2}^{2}),\\
\end{aligned}
\end{equation}
is an averaged system in wind-induced oscillation, where
$\zeta\geq0$ is a damping factor and $\lambda$ is a detuning
parameter (see, e.g. \cite{Guckenheimer}). For convenience, setting
$\zeta=rcos(\theta), \lambda=rsin(\theta), r\geq0,
0\leq\theta\leq\pi/2$, (see \cite{Maclachlan1998}) we write
\eqref{wind} as
\begin{equation}\label{wind2}
\begin{aligned}
\left(\begin{array}{c}\dot{x}_{1}\\ \dot{x}_{2}\end{array}\right)=\left(\begin{array}{cc}-cos(\theta)&-sin(\theta)\\sin(\theta)&-cos(\theta)\end{array}\right)
\left(\begin{array}{c}rx_{1}-\frac{1}{2}sin(\theta)(x_{2}^{2}-x_{1}^{2})-cos(\theta)x_{1}x_{2}\\
rx_{2}-sin(\theta)x_{1}x_{2}+\frac{1}{2}cos(\theta)(x_{2}^{2}-x_{1}^{2})\end{array}\right),
\end{aligned}
\end{equation}
which is of the form \eqref{LNODE}, where
\begin{equation}
Q=\left(\begin{array}{cc}-cos(\theta)&-sin(\theta)\\sin(\theta)&-cos(\theta)\end{array}\right),\quad
M=\left(\begin{array}{cc}r&0\\0&r\end{array}\right),\quad U=-\frac{1}{2}sin(\theta)(x_{1}x_{2}^{2}-\frac{1}{3}x_{1}^{3})+\frac{1}{2}cos(\theta)(\frac{1}{3}x_{2}^{3}-x_{1}^{2}x_{2}).
\end{equation}
Its Lyapunov function
(dissipative case, when $\theta<\pi/2$) or the first integral
(conservative case, when $\theta=\pi/2$) is:
$$H=\frac{1}{2}r(x_{1}^{2}+x_{2}^{2})-\frac{1}{2}sin(\theta)(x_{1}x_{2}^{2}-\frac{1}{3}x_{1}^{3})+\frac{1}{2}cos(\theta)(\frac{1}{3}x_{2}^{3}-x_{1}^{2}x_{2}).$$
{The matrix exponential of the EAVF scheme \eqref{EAVF} for \eqref{wind2}
are calculated by:
\begin{equation*}
\exp(V)=\left(\begin{array}{cc}\exp(-hcr)cos(hsr)&-\exp(-hcr)sin(hsr)\\
\exp(-hcr)sin(hsr)&\exp(-hcr)cos(hsr)\end{array}\right),
\end{equation*}
where $c=cos(\theta), s=sin(\theta)$, and $\varphi(V)$ can be
obtained by $(exp(V)-I)V^{-1}$.} Given the initial values:
\begin{equation*}
x_{1}(0)=0,x_{2}(0)=1,
\end{equation*}
we first integrate the conservative system \eqref{wind2} with the
parameters $\theta=\pi/2, r=20$ and stepsizes $h=1/20\times1/2^{i},
i=-1,\ldots,4$ over the interval $[0,200]$. Setting
$\theta=\pi/2-10^{-4}, r=20,$ we then integrate the dissipative
\eqref{wind2} with the stepsizes $h=1/20\times1/2^{i},
i=-1,\ldots,4$ over the interval $[0,100]$. Numerical errors are
presented in Figs. \ref{windconserv}, \ref{winddissip}. It is noted
that  the integrands appearing in AVF, EAVF are polynomials of
degree  two and the integrands in CRK are polynomials of degree
five. We evaluate the integrals in AVF, EAVF by the 2-point GL
quadrature:
$$b_{1}=\frac{1}{2}, b_{2}=\frac{1}{2},\quad c_{1}=\frac{1}{2}-\frac{3^{\frac{1}{2}}}{6}, c_{2}=\frac{1}{2}+\frac{3^{\frac{1}{2}}}{6},$$
and the integrals appearing in CRK by the 3-point GL quadrature.
Then there is no quadrature error.

The efficiency curves of AVF and MID consist of only five points in
Figs. \ref{windconserv} (a), \ref{windconserv} (b), \ref{winddissip}
(a) (two points overlap in Figs. \ref{windconserv} (a),
\ref{winddissip} (a)), since the fixed-point iterations of MID and
AVF are not convergent when $h=1/10$. Note that $QM$ is
skew-symmetric or negative semi-definite, the convergence of
iterations for the EAVF method is independent of $r$ by Theorem
\ref{iter} and Remark \ref{exam}. Thus larger stepsizes are allowed
for EAVF. The experiment shows that the iterations of EAVF uniformly
work for $h=1/20\times1/2^{i}, i=-1,\ldots,4$. Moreover, it can be
observed from Fig. \ref{windconserv}(d) that MID cannot strictly
preserve the decay of the Lyapunov function.

\begin{figure}[ptb]
\centering
\begin{tabular}[c]{cccc}%
  \subfigure[]{\includegraphics[width=6cm,height=7cm]{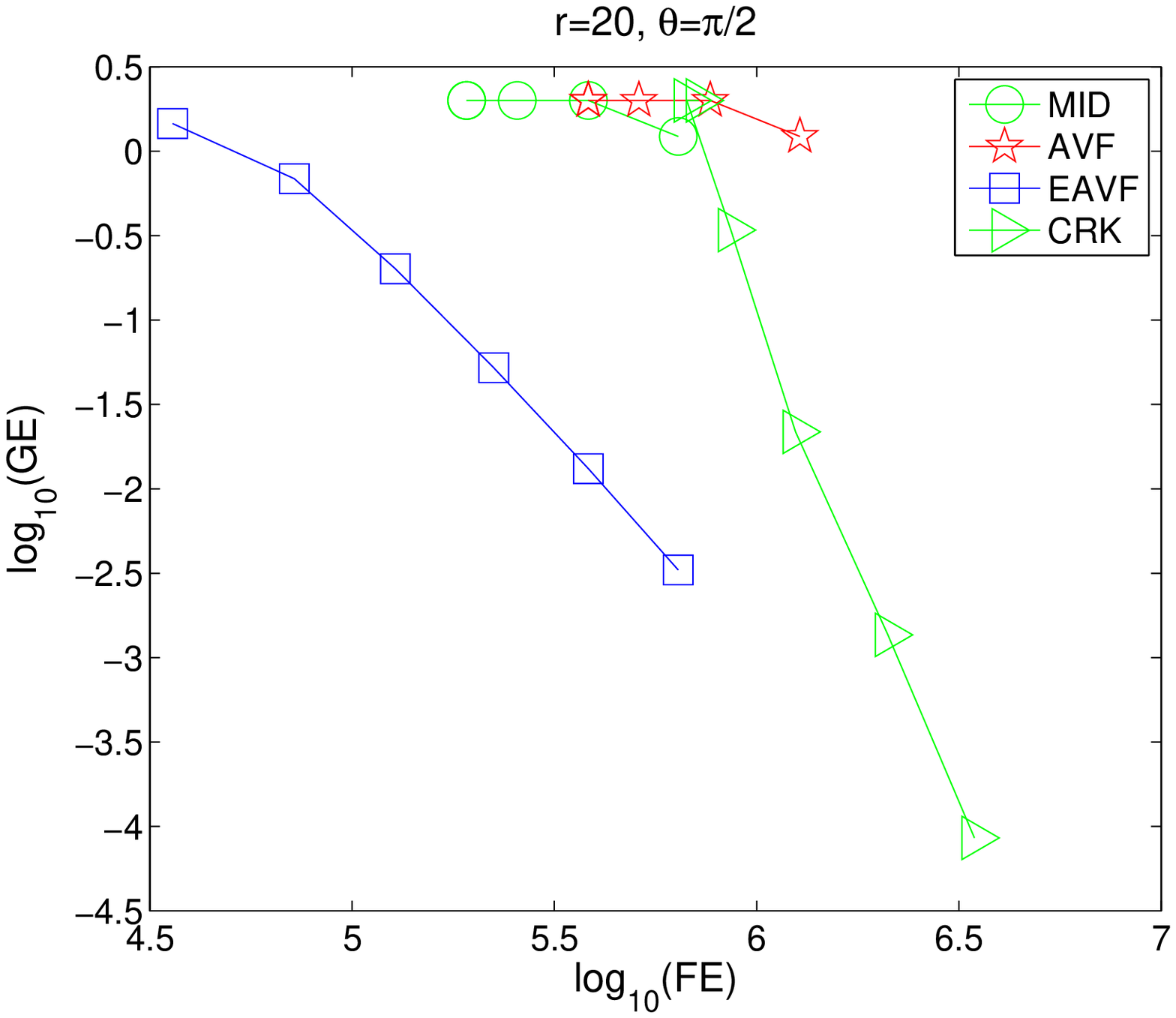}}
  \subfigure[]{\includegraphics[width=6cm,height=7cm]{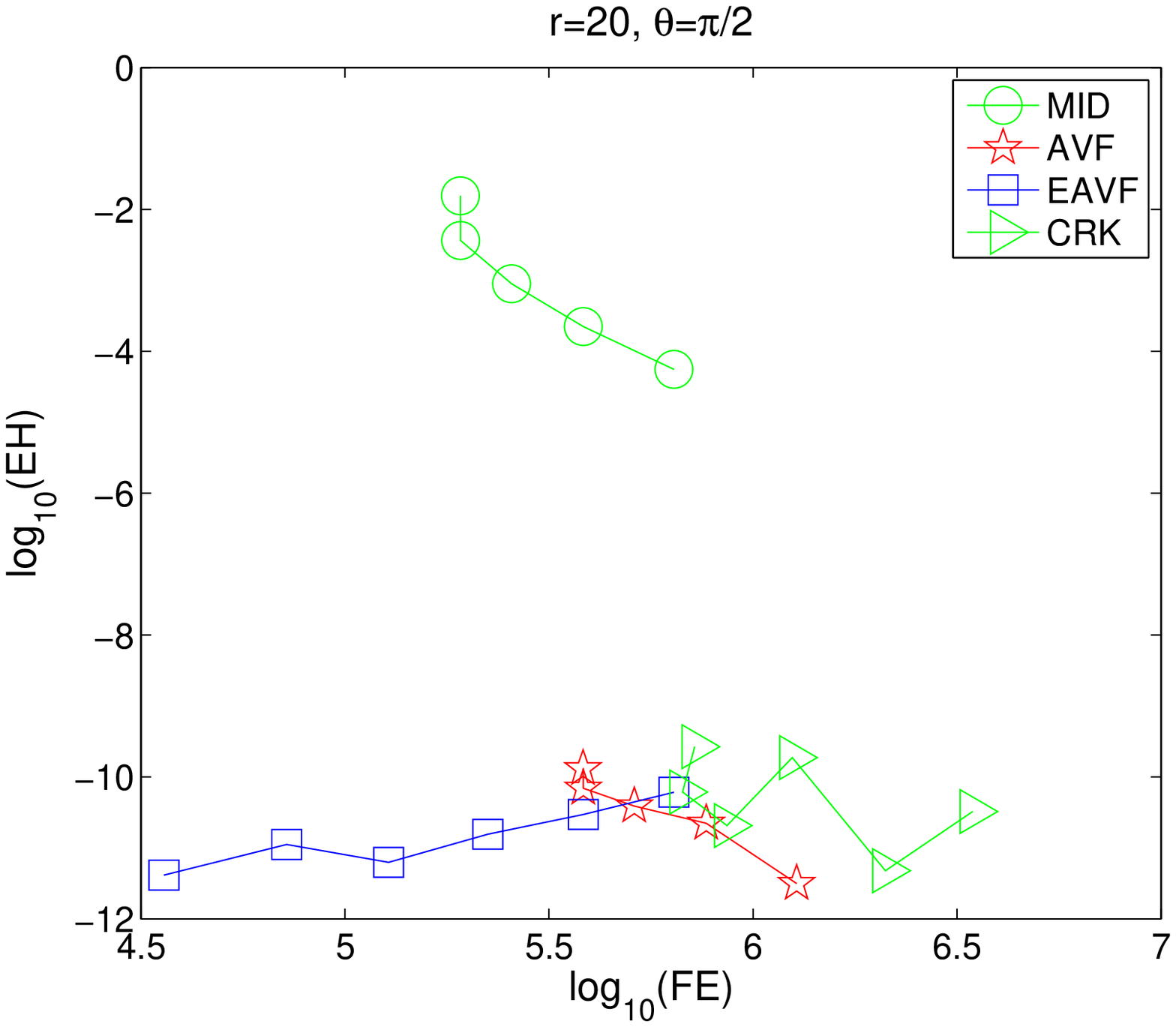}}
\end{tabular}
\caption{Efficiency curves.}
\label{windconserv}
\end{figure}

\begin{figure}[ptb]
\centering
\begin{tabular}[c]{cccc}%
  \subfigure[]{\includegraphics[width=6cm,height=7cm]{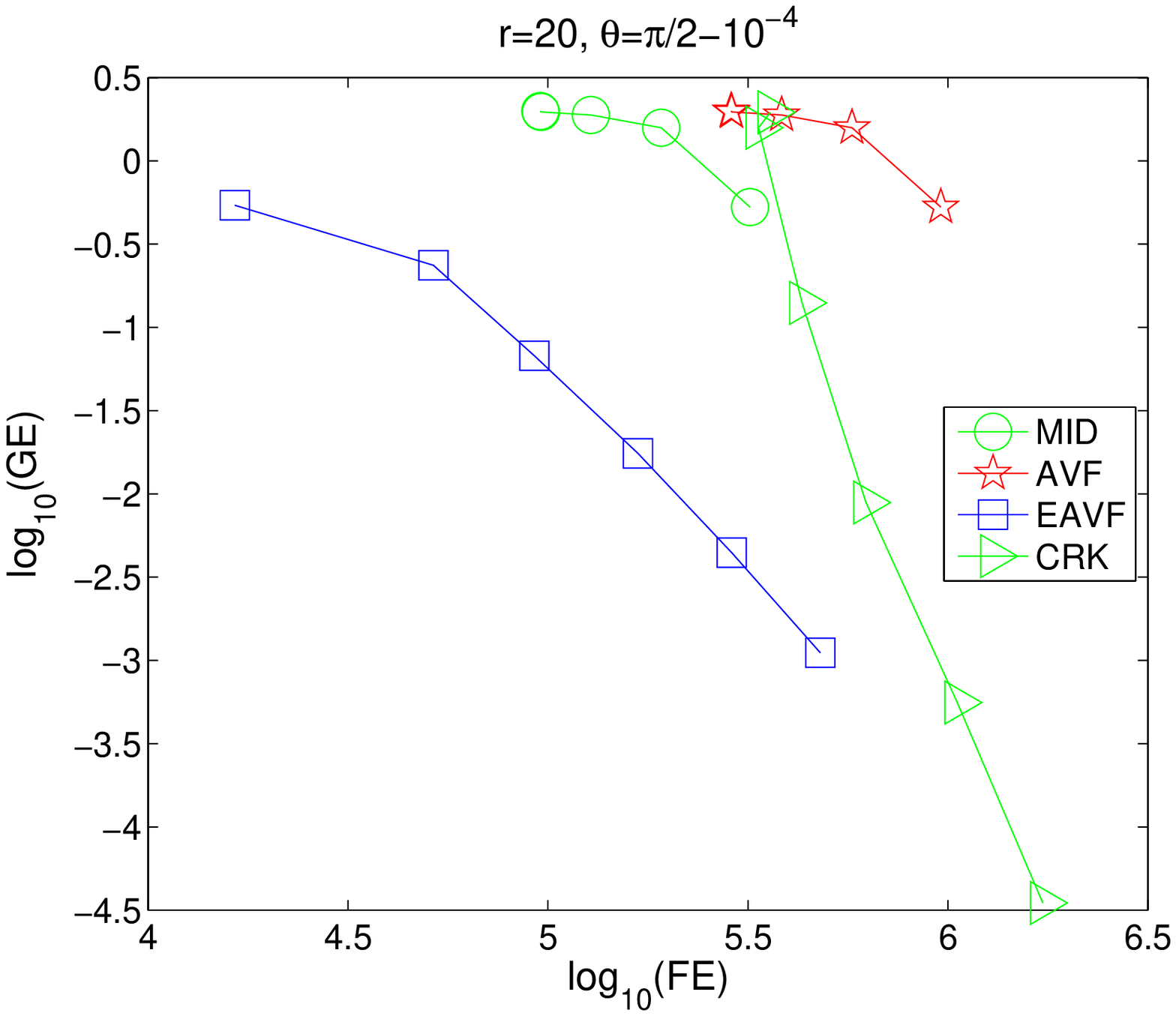}}
  \subfigure[]{\includegraphics[width=6cm,height=7cm]{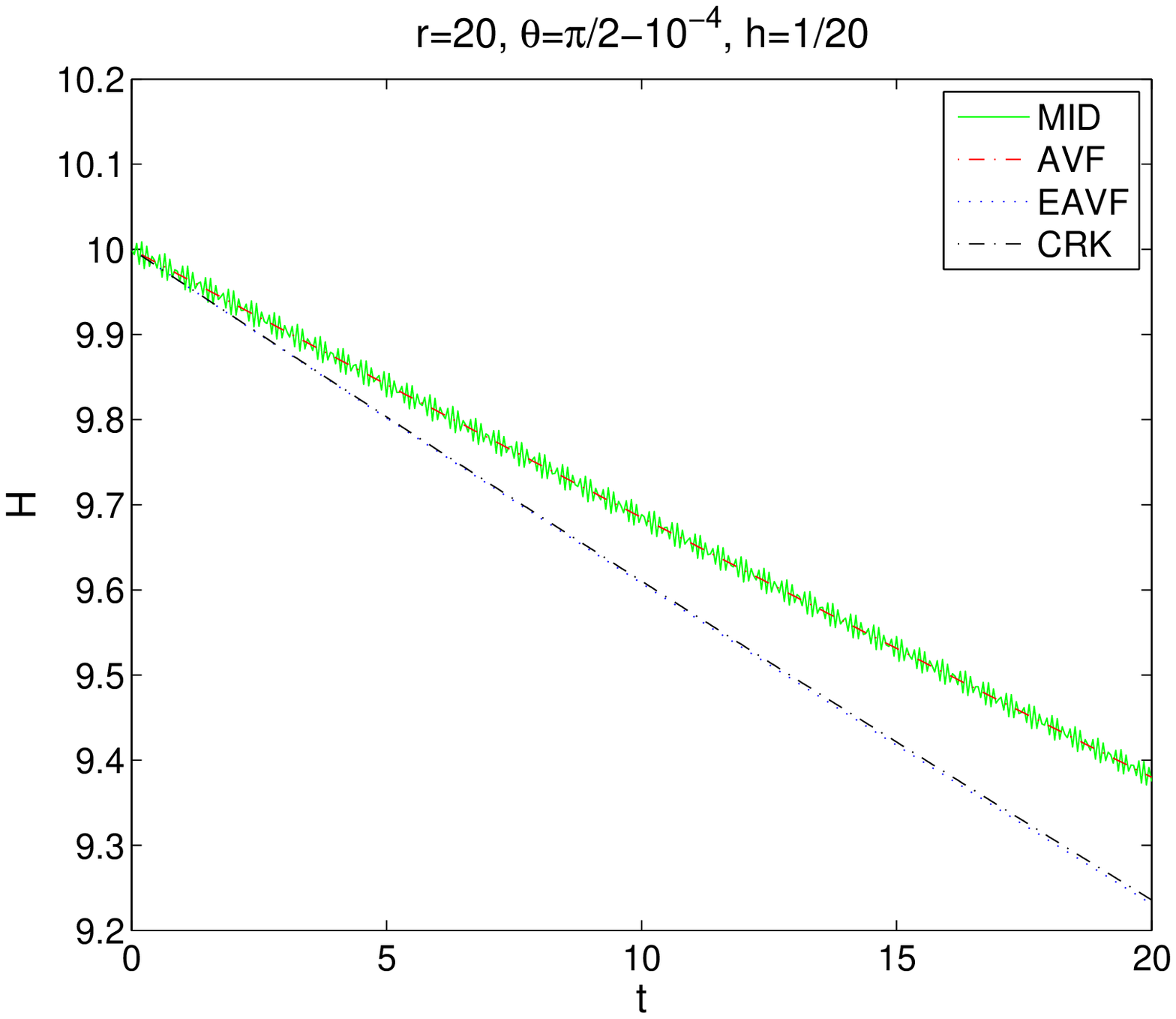}}
\end{tabular}
\caption{(a) Efficiency curves. (b) The Lyapunov function against time $t$.}
\label{winddissip}
\end{figure}

\end{myexp}

\begin{myexp}
The PDE:
\begin{equation}\label{CFPU}
\frac{\partial^{2}u}{\partial t^{2}}=\beta\frac{\partial^{3}u}{\partial t\partial x^{2}}+\frac{\partial^{2}u}{\partial x^{2}}
\left(1+\varepsilon\left(\frac{\partial u}{\partial x}\right)^{p}\right)-\gamma\frac{\partial u}{\partial t}-m^{2}u,
\end{equation}
where $\varepsilon>0, \beta, \gamma\geq0$, is a continuous
generalization of $\alpha$-FPU (Fermi-Pasta-Ulam) system (see, e.g.
\cite{Macias2009}). Taking $\partial_{t}u=v$ and the homogeneous
Dirichlet BC $u(0,t)=u(L,t)=0$, the equation \eqref{CFPU} is of the
type \eqref{TDPDE}, where $X=\mathbf{L}^{2}([0,L])\times
\mathbf{L}^{2}([0,L])$ and
\begin{equation*}
y=\left(\begin{array}{c}u\\v\end{array}\right),\quad\mathcal{Q}=\left(\begin{array}{cc}0&1\\-1&\beta\partial_{x}^{2}-\gamma\end{array}\right),\quad
\mathcal{H}[y]=\int_{0}^{L}\left(\frac{1}{2}u_{x}^{2}+\frac{m^{2}}{2}u^{2}+\frac{v^{2}}{2}+\frac{\varepsilon u_{x}^{p+2}}{(p+2)(p+1)}\right)dx.
\end{equation*}
It is easy to verify that $\mathcal{Q}$ is a negative semi-definite
operator, and thus \eqref{CFPU} is dissipative. {The spatial
discretization yields} a dissipative system of ODEs:
\begin{equation*}
\ddot{u}_{j}(t)-c^{2}(u_{j-1}-2u_{j}+u_{j+1})+m^{2}u_{j}-\beta^{'}(\dot{u}_{j-1}-2\dot{u}_{j}+\dot{u}_{j+1})+\gamma\dot{u}_{j}(t)
=\varepsilon^{'}(V^{'}(u_{j+1}-u_{j})-V^{'}(u_{j}-u_{j-1})),
\end{equation*}
where $c=1/\Delta x, \beta^{'}=c^{2}\beta,
\varepsilon^{'}=c^{p+2}\varepsilon, V(u)=u^{p+2}/[(p+2)(p+1)],
u_{j}(t)\approx u(x_{j},t), x_{j}=j/\Delta x$ for $j=1,\ldots,N-1$
and $u_{0}(t)=u_{N}(t)=0$. {Note that the nonlinear term
$u_{xx}u_{x}^{p}$ is approximated by :
\begin{equation*}
\frac{\partial^{2}u}{\partial x^{2}}\left(\frac{\partial u}{\partial x}\right)^{p}|_{x=x_{j}}=\frac{1}{p+1}\partial_{x}\left(\frac{\partial u}{\partial x}\right)^{p+1}|_{x=x_{j}}
\approx\frac{1}{p+1}\left(\left(\frac{u_{j+1}-u_{j}}{\Delta x}\right)^{p+1}-\left(\frac{u_{j}-u_{j-1}}{\Delta x}\right)^{p+1}\right)/\Delta x.
\end{equation*}}

We now write it in the compact form \eqref{2th}:
\begin{equation*}
\ddot{q}-N\dot{q}+\Omega q=-\nabla U_{1}(q),
\end{equation*}
where $q=(u_{1},\ldots,u_{N-1})^{\intercal}, N=\beta^{'}D-\gamma I, \Omega=-c^{2}D+m^{2}I,
U_{1}(q)=\varepsilon^{'}\sum_{j=0}^{N-1}V(u_{j+1}-u_{j})$ and
\begin{equation*}
D=\left(\begin{array}{ccccc}-2&1& & & \\1&-2&1& & \\&\ddots&\ddots&\ddots& \\ & &1&-2&1\\& & &1&-2\end{array}\right).
\end{equation*}
In this experiment, we set $p=1, m=0, c=1, \varepsilon=\frac{3}{4},$
and $\gamma=0.005.$ Consider the initial conditions in
\cite{Macias2009}:
\begin{equation*}
\phi_{j}(t)=B\ln\left\{\left(\frac{1+\exp[2(\kappa(j-97)+t\sinh(\kappa))]}{1+\exp[2(\kappa(j-96)+t\sinh(\kappa))]}\right)
\left(\frac{1+\exp[2(\kappa(j-32)+t\sinh(\kappa))]}{1+\exp[2(\kappa(j-33)+t\sinh(\kappa))]}\right)\right\}
\end{equation*}
with $B=5, \kappa=0.1$, that is,
\begin{equation*}
\left\{\begin{aligned}
&u_{j}(0)=\phi_{j}(0),\\
&v_{j}(0)=\dot{\phi}_{j}(0).\\
\end{aligned}\right.
\end{equation*}
for $j=1,\ldots,N-1$. Let $N=128, \beta=0, 2.$ We compute the
numerical solution by MID, AVF and EAVF with the stepsizes
$h=1/2^{i}, i=1,\ldots,5$ over the time interval $[0,100]$.
{Similarly to EAVF \eqref{EAVF1}, the nonlinear systems resulting
from MID \eqref{MID} and AVF \eqref{AVF} can be reduced to:
\begin{equation*}
q^{1}=q^{0}+hp^{0}+\frac{h}{2}N(q^{1}-q^{0})-\frac{h^{2}}{4}\Omega(q^{1}+q^{0})-\frac{h^{2}}{2}\nabla U_{1}(\frac{q^{0}+q^{1}}{2}),
\end{equation*}
and
\begin{equation*}
q^{1}=q^{0}+hp^{0}+\frac{h}{2}N(q^{1}-q^{0})-\frac{h^{2}}{4}\Omega(q^{1}+q^{0})-\frac{h^{2}}{2}\int_{0}^{1}\nabla U_{1}((1-\tau)q^{0}+\tau q^{1})d\tau
\end{equation*}
respectively. Both the velocity $p^{1}$ of MID and AVF can be recovered by
\begin{equation*}
\frac{q^{1}-q^{0}}{h}=\frac{p^{1}+p^{0}}{2}.
\end{equation*}}
The integrals in AVF and EAVF are exactly evaluated by the 2-point GL quadrature.
Since $\exp(hA), \varphi(hA)$  in \eqref{EAVF1} have no explicit
expressions, they are calculated by the Matlab package in
\cite{Berland2007}. The basic idea is evaluating $\exp(hA),
\varphi(hA)$ by their Pad\'{e} approximations. Numerical results are
plotted in Figs. \ref{FPU}. Alternatively, there are other popular
algorithms such as contour integral method and Krylov subspace
method for matrix exponentials and $\varphi$-functions. Readers are
referred to \cite{Hochbruck2010} for a summary of algorithms and
well-established mathematical software.

According to Theorem \ref{iter2}, the convergence of iterations in
the EAVF scheme is independent of $\Omega$ and $N$. Iterations of
MID and AVF are not convergent when $\beta=2, h=1/2$. Thus the
efficiency curves of MID and AVF in Fig. \ref{FPU}(b) consist of
only $4$ points. From Fig. \ref{FPU}(c), it can be observed that the
EAVF method can preserve dissipation even using the relatively large
stepsize $h=1/2$.

\begin{figure}[ptb]
\centering
\begin{tabular}[c]{cccc}%
  \subfigure[]{\includegraphics[width=4.5cm,height=7cm]{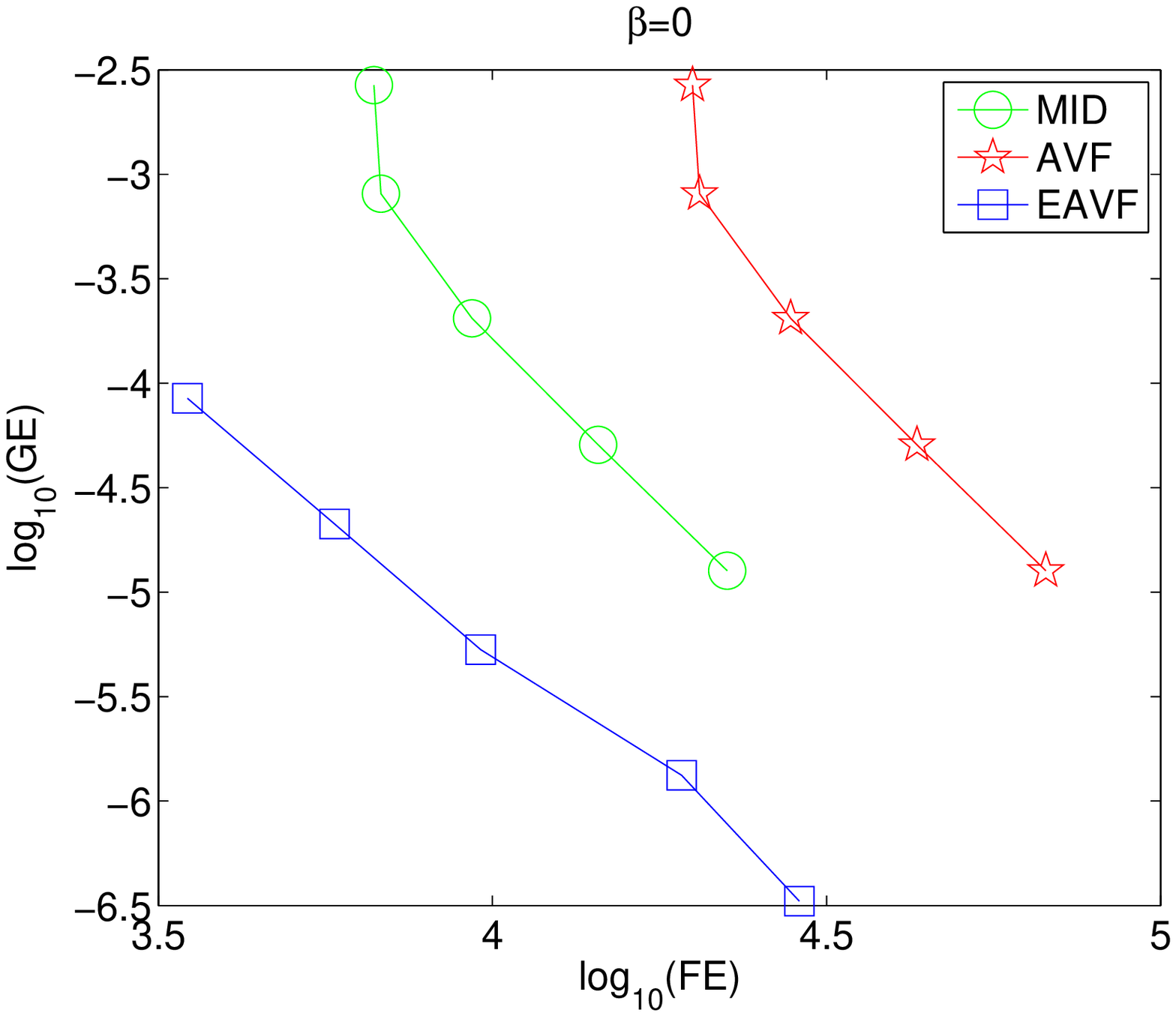}}
  \subfigure[]{\includegraphics[width=4.5cm,height=7cm]{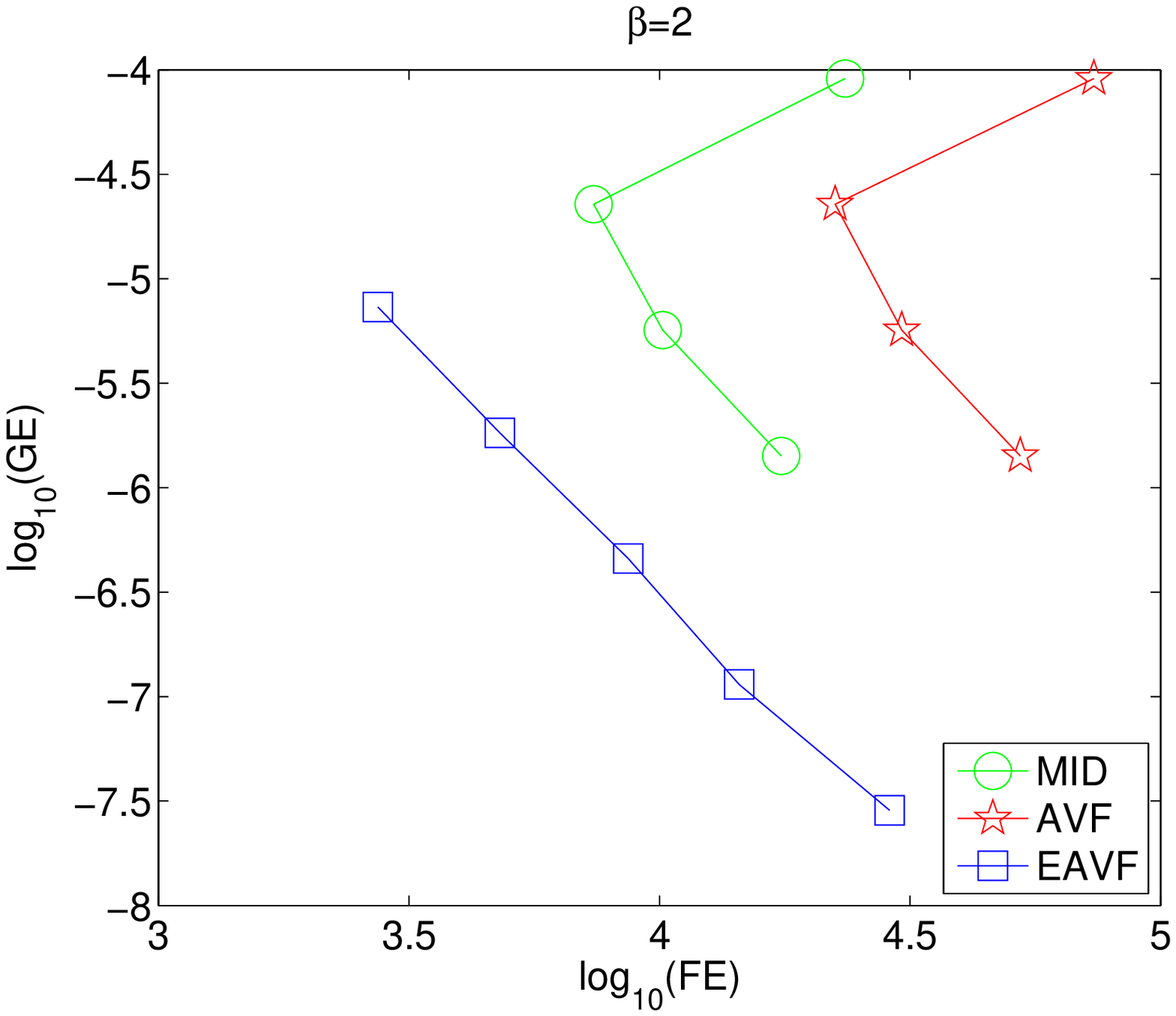}}
  \subfigure[]{\includegraphics[width=4.5cm,height=7cm]{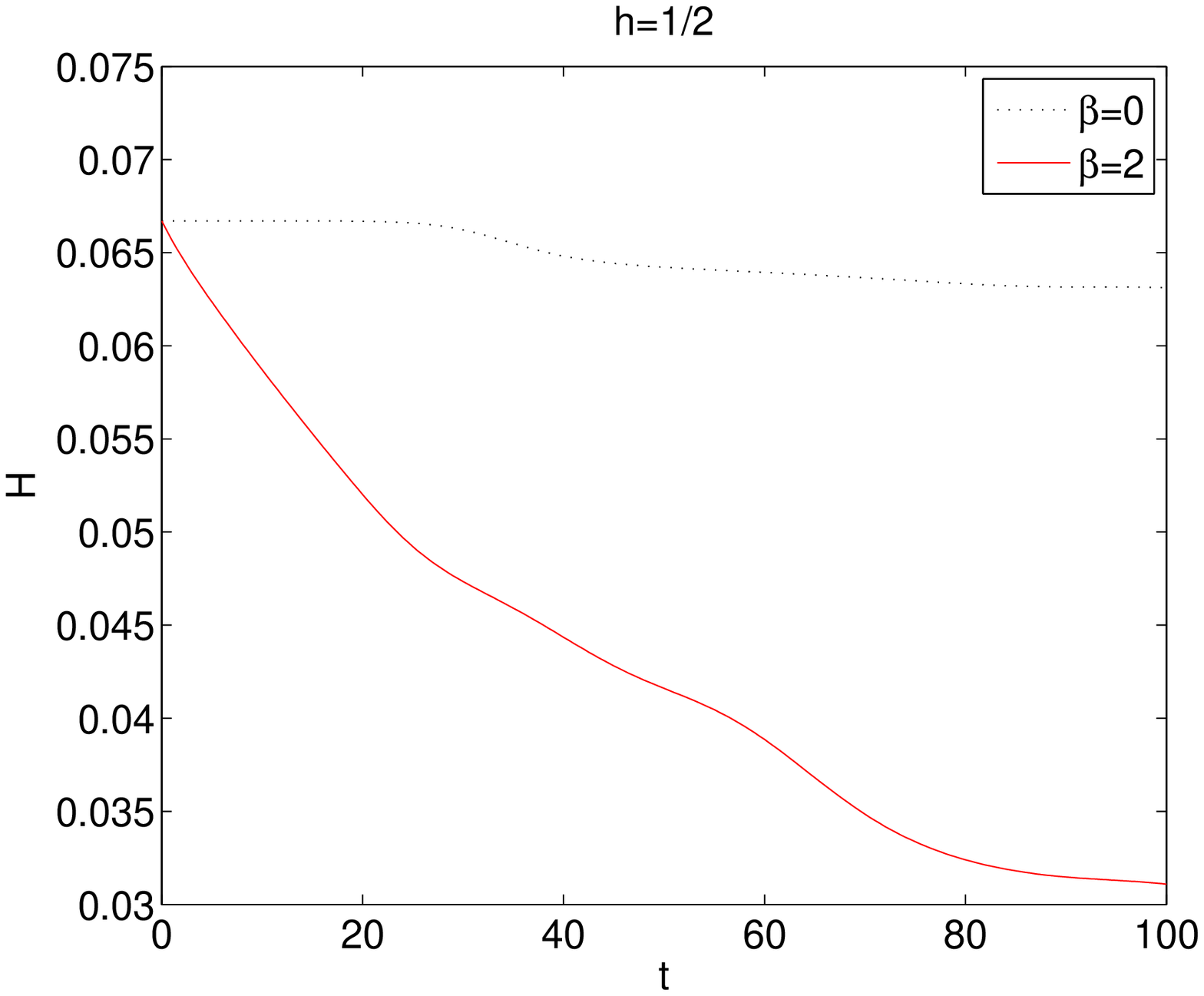}}
\end{tabular}
\caption{(a)\ (b)
 Efficiency curves. (c) The decay of Lyapunov
function obtained by EAVF.}

\label{FPU}
\end{figure}

\end{myexp}

\newpage
\section{Conclusions}
Exponential integrators can be traced back to the original paper by
Hersch \cite{Hersch1958}. The term `exponential integrators' was
coined in the seminal paper by Hochbruck, Lubich and Selhofer
\cite{Hochbruck1998}. It turns out  that exponential integrators
have constituted an important class of schemes for the numerical
simulation of differential equations. In this paper, combining the
ideas of the exponential integrator with the average vector field,
we  derived and analyzed a new exponential scheme EAVF preserving
the first integral or the Lyapunov function for the conservative or
dissipative system \eqref{LNODE}, which includes numerous important
mathematical models in applications. The symmetry of EAVF ensures
the prominent long-term numerical behavior. Due to the implicity of
EAVF requires iteration solutions, we  analysed the convergence of
the fixed-point iteration and showed that the convergence is free
from the influence of a wide range of coefficient matrices $M$. In
the dynamics of the triatomic molecule, the wind-induced oscillation
and the damped FPU problem, we compared the new EAVF method with the
MID, AVF and CRK methods. The three problems are modeled by the
system \eqref{LNODE} having a dominant linear term and small
nonlinear term. In the aspects of algebraic accuracy as well as
preserving energy and dissipation, EAVF is very efficient among the
four methods. In general, energy-preserving and energy-decaying
methods are implicit, and then iteration solutions are required.
With a relatively large stepsize, the iterations of EAVF are
convergent, whereas AVF and MID do not work in experiments.
Therefore, EAVF is expected to be a promising method solving the
system \eqref{LNODE} with $||QM||\gg||QHess(U)||$.

\subsubsection*{Acknowledgments.}
The authors are sincerely thankful to two anonymous referees for
their valuable suggestions, which help improve the presentation of
the manuscript.

\end{document}